\documentclass[12pt]{article}
\usepackage{latexsym}
\usepackage{amssymb}
\usepackage{amsfonts}
\usepackage{amsmath}
\usepackage{amsthm}
\usepackage{ulem}		
\usepackage{sectsty}	
\usepackage[all]{xy}		
\usepackage{bm}		
\usepackage{wasysym}	
\usepackage{ifthen}
\usepackage{graphics}
\usepackage{psfrag, psfig}
\usepackage[dvips]{graphicx}
\usepackage{comment}

\usepackage[pdftex]{lscape} 
\usepackage{rotating} 
\usepackage{epsfig}

\newcounter{bean}
\newcounter{beany}
\renewenvironment{enumerate}{\begin{list}{(\roman{bean})}{\usecounter{bean}\setlength{\rightmargin}{\leftmargin}}}{\end{list}\setcounter{bean}{1}}
\newenvironment{a-enumerate}{\begin{list}{(A-\arabic{beany})}{\usecounter{beany}\setlength{\rightmargin}{\leftmargin}}}{\end{list}}
\newenvironment{b-enumerate}{\begin{list}{(B-\arabic{bean})}{\usecounter{bean}\setlength{\rightmargin}{\leftmargin}}}{\end{list}\setcounter{bean}{1}}
\newenvironment{d-enumerate}{\begin{list}{(D-\arabic{bean})}{\usecounter{bean}\setlength{\rightmargin}{\leftmargin}}}{\end{list}\setcounter{bean}{1}}

\theoremstyle{plain}

\newtheorem{thm}{Theorem}[section]
\newtheorem{lem}[thm]{Lemma}
\newtheorem{prop}[thm]{Proposition}
\newtheorem{cor}[thm]{Corollary}

\theoremstyle{definition}
\newtheorem{defn}[thm]{Definition}

\newtheorem*{claim}{Claim}

\newtheorem{rmk}[thm]{Remark}

\numberwithin{equation}{section}

\newcommand{\cc}{\ensuremath{\Subset}}
\newcommand{\del}{\ensuremath{\partial}}

\DeclareMathOperator{\jac}{\ensuremath{\mathrm{Jac}}}
\DeclareMathOperator{\tr}{\ensuremath{\mathrm{tr}}}

\newcommand{\R}{\ensuremath{\mathcal{R}}} 
\newcommand{\RH}{\ensuremath{\mathcal{R}}} 

\newcommand{\RR}{\ensuremath{\mathbb{R}}} 
\newcommand{\CC}{\ensuremath{\mathbb{C}}} 
\newcommand{\NN}{\ensuremath{\mathbb{N}}} 
\newcommand{\ZZ}{\ensuremath{\mathbb{Z}}} 

\newcommand{\ii}{\ensuremath{\mathit{\iota}}}

\newcommand{\bigo}{\ensuremath{\textup{O}}}

\renewcommand{\emph}[1]{\textit{#1}}

\newcommand{\word}[2]{\ensuremath{\mathbf{#1}#2}}

\newcommand{\Dom}{\ensuremath{\mathrm{Dom}}}
\newcommand{\id}{\ensuremath{\mathrm{id}}}

\newcommand{\Diffeo}{\ensuremath{\mathrm{Diff}}}
\newcommand{\End}{\ensuremath{\mathrm{End}}}

\newcommand{\e}{\ensuremath{\varepsilon}}

\newcommand{\oo}[2]{\ensuremath{#2^{#1}}}

\newcommand{\Cantor}{\ensuremath{\mathcal O}} 

\newcommand{\U}{\ensuremath{\mathcal{U}}}
\renewcommand{\H}{\ensuremath{\mathcal{H}}}
\newcommand{\I}{\ensuremath{\mathcal I}} 

\newcommand{\Omegax}{\ensuremath{\Omega_x}}

\newcommand{\MT}{\ensuremath{\Psi}}

\newcommand{\AAA}{\ensuremath{\textup{\textbf{\textsf{A}}}}}
\newcommand{\BBB}{\ensuremath{\textup{\textbf{\textsf{B}}}}}
\newcommand{\DDD}{\ensuremath{\textup{\textbf{\textsf{D}}}}}
\newcommand{\UUU}{\ensuremath{\textup{\textbf{\textsf{U}}}}}
\newcommand{\VVV}{\ensuremath{\textup{\textbf{\textsf{V}}}}}
\newcommand{\WWW}{\ensuremath{\textup{\textbf{\textsf{W}}}}}

\newcommand{\boldkappa}{\ensuremath{\boldsymbol{\kappa}}}

\title{Infinitely Many Moduli of Stability at the Dissipative Boundary of Chaos}
\author{P. Hazard, M. Martens, and C. Tresser}
\date{\today}

\begin{document}







\maketitle


\begin{abstract}
In the family of area-contracting H\'enon-like maps 
with zero topological entropy we show that there are 
maps with infinitely many moduli of stability. 
Thus one cannot find all the possible topological 
types for non-chaotic area-contracting H\'enon-like 
maps in a family with finitely many parameters. 
A similar result, but for the chaotic maps in the family, 
became part of the folklore a short time after H\'enon 
used such maps to produce what was soon conjectured 
to be the first non-hyperbolic strange attractor in $\RR^2$.  
Our proof uses recent results about infinitely renormalisable 
area-contracting H\'enon-like maps; 
it suggests that the number of parameters needed to 
represent all possible topological types for area-contracting 
H\'enon-like maps whose sets of periods of their periodic orbits are finite (and in particular are equal to $\{1,\, 2,\dots,\,2^{n-1}\}$ or an initial segment of this $n$-tuple) increases with the number of periods. In comparison, among $C^k$-embeddings of the 2-disk with $k\geq 1$, the maximal moduli number for non-chaotic but non area-contracting maps in the interior of the set of zero-entropy is infinite.
\end{abstract}

\section{Introduction}
In this paper we prove the following theorem:
\begin{thm}[Zero entropy H\'enon-like maps have infinite modulus]\label{thm:InfPar4Henon}
The area-contracting H{\'enon}-like maps with zero topological entropy form a family of diffeomorphisms with infinitely many moduli. In particular, infinitely many parameters are needed to exhaust all the possible topological types, even if one only considers the non-chaotic part of that family.
\end{thm}

Here chaos means topological chaos, or in more technical terms, positive topological entropy. Before commencing we will need a few more words to explain this result, 
and why it is important in the theory of dynamical systems. We will not get into the issue of which parts of the boundary of zero entropy, in the set of continuous self-maps on a space, is made of maps with zero entropy, a deep issue that depends on the smoothness class and on dimension (for a review, see, {\it e.g.}, \cite{MilTre}).

In 1978 Jacob Palis~\cite{Palis} discovered a pair of new invariants of topological conjugacy, one for flows and one for diffeomorphisms and embeddings on two-dimensional manifolds.  Assuming $p_u$ and $p_s$ are saddle points (critical points if we consider a flow $\{\phi _t\}$ , periodic points if we consider a diffeomorphism or an embedding $F$) we further assume that there is a tangency (of any contact order -- the manifolds may even coincide) between the unstable manifold $W^u(p_u)$ of $p_u$  and the stable manifold $W^s(p_s)$ of $p_s$. We note it is possible that $p_u=p_s$. 


We will only be concerned with diffeomorphisms in this paper. Incidentally, we notice that for flows one gets back to the diffeomorphism case by considering, {\it e.g.}, the time one map $F= \phi _1$ that, in particular, turns the critical points of $\{\phi _t\}$ (as well perhaps as some periodic orbits of the flow) into fixed-points of $F$.

The Palis invariant is the real number
\begin{equation}
P=P_{F:p_u,p_s}=\frac{\log |\lambda|}{\log|\mu|}
\end {equation}
where $\lambda\in (-1,1)$ is the stable eigenvalue of (the linearised maps at) $p_u$ and $\mu\in (-\infty,-1)\cup (1,\infty)$ is the unstable eigenvalue of (the linearised maps at)  $p_s$. In order to simplify the discussion, we assume from now on that all eigenvalues of the linearised maps near $p_u$  and $p_s$ are positive, without loss of generality since otherwise we can consider $G=F^2$ and remember the original signs.

Assume now that no special condition  is imposed that may generate constraints that would limit the possible values of  $P$ (such as, {\it e.g.}, reversibility if $p_u=p_s=O$ in the case of a flow such that $O$ is invariant under the reversibility symmetry).  Then the Palis invariant can be varied continuously, effectively giving rise to an \emph{arc} ({\it i.e.}, a continuous one-parameter family) of distinct topological types parametrised by $P$.  Such a situation, when it occurs for some numerical invariant $I$, is often referred to by saying that there is a \emph{modulus (of topological conjugacy)} associated with the invariant  $I$; one can then also say that the modulus is attached to the map.

More generally, one says that the number of moduli is finite when a finite number of numerical invariants can be prescribed independently of each other and this number is maximal, in that any additional numerical invariant will not independently vary with respect to the others. Otherwise we say the number of moduli is infinite.

In some contexts one might be interested in studying topological invariants in continuous families of homeomorphisms or diffeomorphisms, but since the Palis invariant is made out of smooth data, we will only consider $C^1$-families of $C^k$-diffeomorphisms with $k\geq 1$, so that at least the maps and their derivatives each vary continuously.  Accordingly, the number of moduli of a map, if finite, is the minimal number of parameters needed to get representatives of all the \emph{topological types} ({\it i.e.}, the classes of topological conjugacy) in a $C^1$-neighborhood of the map.  If the number of moduli is infinite, one needs infinitely many parameters, or one can also say that no $C^1$-family depending on finitely many parameters will contain all the topological types existing in any $C^1$-neighborhood of the map.

Since the Palis invariant is expressed as a function of two eigenvalues, it looks \emph{a priori} like a smooth invariant.  If it was only that, there would be no interest in this ratio since all eigenvalues are themselves invariant under smooth conjugacy, {\it i.e.}, under smooth changes of variables (by a trivial application of the chain rule).  Thus, the topological character is both what makes the Palis invariant important and what makes it surprising.  

Anyone who has learned about this unexpected invariant and knows about the \emph{Newhouse phenomenon} ({\it i.e.}, the abundance of non-degenerate tangencies under mild conditions between some stable and unstable manifolds as described in Newhouse's thesis and reported in \cite{Newhouse}) would immediately guess that it is possible to construct a diffeomorphism with infinitely many moduli by assembling these two ingredients.  As the reader may know or have guessed, Palis indicated how to proceed to get such an example in the same paper~\cite{Palis} where he first reported on this invariant.  He indicated how to construct examples of maps whose complete topological unfolding cannot be contained in a family of maps that depends on only finitely many parameters.  Roughly speaking, in \cite{Palis} Palis uses the theory of Newhouse inductively to construct infinitely many simultaneous tangencies at successively smaller and smaller scales, in such a way that each tangency carries a modulus that is independent of the previous ones.  
From there one gets, without too much effort, that the family of H\'enon-like maps 
with positive topological entropy (see, {\it e.g.}, in \cite{dCLM1},~\cite{LM1},~\cite{LM2})
 contains maps having infinitely many moduli, a folklore fact that has been announced by many.

To give proper credit on the issue of having infinite modality we point out that, as reported in \cite{Palis}, an earlier example was constructed by Robinson and Williams~\cite{RandW} using different ideas. 
One indispensable feature in the dealings with both examples is that they must have positive topological entropy. Also notice that a theorem of de Melo and van Strien~\cite{dMvS3} gives necessary and sufficient conditions for the presence of finitely many moduli in the closure of Axiom A diffeomorphisms satisfying the no-cycles condition. 
This leaves open the question: 

\vspace{5pt}

\noindent
\emph{{\bf Q1:} ``Does there exist a zero topological entropy embedding of the closed 2-disk with infinitely many moduli?"} 

\vspace{5pt}

In other words: does there exist a non-chaotic diffeomorphism ({\it i.e.}, one with zero topological entropy) whose complete topological unfolding cannot be produced by any family of maps that would depend only on finitely many parameters?
Since maps with complicated dynamics have many topological features that might change independently, it seems likely that simpler dynamics (in particular, topological entropy zero) makes it more difficult to have examples requiring infinitely many parameters to unfold all possible topological types.  Yet it is known at least since Zehnder's Theorem on homoclinic orbits~\cite{Zehnder}, that \emph{regular behavior can be easily perturbed to chaotic behavior for conservative systems}, and we will see below that the above question has an easy positive answer based on ideas that are at least implicit in~\cite{Harrison}.

Thus, a better question seems to be: 

\vspace{5pt}

\noindent
\emph{{\bf Q2:} ``Does there exist an area-contracting zero topological entropy embedding of the closed 2-disk with infinitely many moduli?"} 

\vspace{5pt}

Another reason to prefer question {\bf Q2} is the importance of attractors not only in dynamics but also in its applications to various scientific disciplines.  In this paper we construct examples of families of area-contracting embeddings of the two-disk (and more precisely \emph{uniformly area-contracting embeddings} that contract volume with a definite rate bounded from above by some positive $\rho<1$) that have infinitely many moduli but zero entropy (in fact these maps are on the boundary of chaos, as we shall see).  The area contraction hypothesis is, as we explain below, what makes the problem somewhat non-trivial. 

Recall that an embedding $F$ is \emph{Kupka-Smale} (or \emph{KS} for short) if all periodic points are hyperbolic and each intersection between the invariant manifolds of those periodic points is transverse. We say that $F$ is 
\emph{$\Omega$-Kupka-Smale} (or \emph{$\Omega$-KS} for short) if all the periodic orbits of $F$ are hyperbolic, and hence would (individually) survive any $C^1$-small enough perturbation of $F$.   If there are infinitely many periodic orbits, it might still be the case that an arbitrarily small perturbation destroys some of them: KS describes the most obvious necessary conditions for \emph{structural stability} (the property that maps $C^1$-near some map $G$ are topologically conjugate to $G$), yet KS is not enough to guarantee structural stability.

Notice that, by definition, structurally stable diffeomorphisms have a zero number of moduli, so one may expect that the $\Omega$-KS property makes a positive answer to {\bf Q2} more unlikely, and this might be the case in the conservative setting.  However, the area-contracting examples that we construct are $\Omega$-KS. The KS version of the maps that we consider were used in~\cite{GST} to build the first example of a $C^\infty$-KS-diffeomorphism of the 2-sphere without sources or sinks. Obviously, only the $\Omega$-KS part of KS can be imposed when considering the Palis invariant.  

As previously mentioned, it had been known for some time that the Palis invariant causes H\'enon-like maps to depend on infinitely parameters when the topological entropy is positive. 
Thus, except to point out where the novelty of our results lie, the ``boundary of chaos'' (boundary of zero entropy) does not need to be mentioned in the statement of our theorem.  
Yet, being at the boundary of chaos is indispensable in the construction that we will present and which differs considerably from the methods used for chaotic maps. 

Let us now come back to {\bf Q1} and in fact abandon the $\Omega$-KS and the area-contracting conditions.  Consider then an annulus map $F_A$ that leaves the circles $C_r$ of constant radius $r$ invariant, and twists each such circle $C_r$ by an angle $\theta(r)$ so that the smooth function $\theta(r)$ has infinitely many maxima and minima: the values of $\theta(r)$ at these extrema are topological invariants and provide the moduli that we are looking for. 
The $\theta(r)$-controlled examples have quite mild dynamics but they are nevertheless at the boundary of chaos in the conservative case, at least in the real-analytic case where this follows from a classical result by Zehnder \cite{Zehnder}.  So we arrive at another question that has been even more well-circulated than {\bf Q2} although it has, as we shall see, a simple answer.

\vspace{5pt}

\noindent
\emph{{\bf Q3:} ``Does there exist a diffeomorphism of the plane, or a part of it, with infinitely many moduli and which lies in the interior of the set of maps with zero topological entropy?"}. 

\vspace{5pt}

The answer is ``yes" since examples are easily constructed in the non-conservative version of the annulus maps that we have just mentioned. \emph{Normal hyperbolicity} (contraction toward or expansion from the invariant circles) allows us to have stably invariant circles on which rotation numbers can be varied at will but in a controlled way. The only thing thats needs to be checked is that we can vary the rotation numbers on infinitely many invariant curves, and when they accumulate this variation is sufficiently smooth.

So the interesting problem from this line of questioning seems to be: 

\vspace{5pt}

\noindent
\emph{{\bf Q4:} ``Does there exist a diffeomorphism or embedding of the closed 2-disk that contracts volume with a ratio bounded from above by some $\rho <1$ and which has infinitely many moduli but lies in the interior of the set of maps with zero topological entropy?"}. 

\vspace{5pt}

This question we will leave open as we suspect that its solution would need some new ideas.  We conjecture that the answer is ``no" (and  will use the fact that the truth of a conjecture is quite often much less relevant than the formulation of the corresponding question).
 
In this paper, as in \cite{GST}, \cite{dCLM1},~\cite{LM1},~\cite{LM2}, by \emph{H\'enon-like maps} we mean a particular class of maps that resemble the 2-parameter maps introduced by Michel H{\'e}non in 1976 to give an early numerical example of what seemed to be a non-hyperbolic strange attractor in the plane. H\'enon-like maps are a prototype of horseshoe forming maps when one varies some parameter(s), hence their importance in (low dimensional) dynamics.  Our construction answering {\bf Q2} uses the fine knowledge of the structure of strongly dissipative H\'enon-like maps at the accumulation of a cascade of period doubling bifurcations that has recently been reported in ~\cite{dCLM1},~\cite{LM1}, and ~\cite{LM2}.  Our construction uses infinitely renormalisable H\'enon-like maps that are very dissipative and at the boundary of chaos (also at the boundary of the Morse-Smale diffeomorphisms), hence zero entropy maps that possess periodic orbits whose set of periods is exactly the set of powers of 2.  Successive but not contiguous pairs of these orbits permit us to build independently varying Palis invariants. As we can do this an arbitrary number of times we find the infinite number of moduli with zero entropy that we seek. 

The rest of the paper is organised as follows. 
In Section~\ref{sect:prelim} we recall the relevant definitions and other ingredients necessary for our construction. 
The construction is then presented in Section~\ref{sect:mainconstr}. 
More precisely, in Subsection~\ref{subsect:1stconstr} we construct families of infinitely renormalisable H\'enon-like maps. 
For a fixed parameter value the H\'enon-like map possesses a prescribed collection of heteroclinic tangencies. 
Then in Subsection~\ref{subsect:2ndconstr} we use these markings to inductively construct families with arbitrarily many moduli which we prove to be independent of one another.
Finally, in Subsection~\ref{subsect:tangfamily} we use this second family to construct a \emph{tangency family}, {\it i.e.}, a family where the marked tangencies persist for all parameters (see Section~\ref{sect:tang+full} for more details).

\section{Preliminaries}\label{sect:prelim}
\subsection{Notations and Conventions}\label{subsect:notations}
Let $\pi_x,\pi_y\colon \RR^2\to \RR$ denote the projections onto the $x$- and $y$-coordinates respectively. 
We will identify these with their extensions to $\CC^2$. 
Given points $a,b\in \RR$ we will denote the closed interval between $a$ and $b$ by $[a,b]=[b,a]$.

Given a diffeomorphism $F$ of the plane (or a surface) 
with a periodic saddle $p$ for any distinct pair of points 
$r_0, r_1\in W^u(p)$ we denote the minimal closed subarc 
of $W^u(p)$ containing $r_0$ and $r_1$ by $[r_0,r_1]^u$. 
For $r_0,r_1\in W^s(p)$ define $[r_0,r_1]^s$ likewise. 
Given a closed topological disk $D$ whose boundary consists 
of subarcs of stable and unstable manifolds let $\del^u D$ 
denote the union of closed boundary subarcs lying in unstable manifolds 
and $\del^s D$ denote the union of closed boundary subarcs lying in stable manifolds.

Finally given a function $F$ we will denote the derivative by $DF$ and the derivative with respect to the variable $b$, $x$, $y$, etc. by $\partial_bF$, $\partial_xF$, $\partial_yF$, etc.
Given a map $F$ of the $(x,y)$-plane, depending upon a parameter $b$, we will occasionally denote the derivative with respect to the spacial variables $x$ and $y$ by $\partial_{x,y}F$.

%
%
%
\subsection{Unimodal Maps}\label{subsect:unimodal}
In this and all following sections we will adopt the notation and terminology from~\cite{Haz1} 
with minor simplifications stated below, which we can make as we will only consider period-doubling combinatorics.
Let $r\in\{3,\ldots,\infty\}$. 
Denote by $\U^r$ the space of maps $f\in \End^r([0,1])$ satisfying the following properties:
\begin{enumerate}
\item 
$f$ has a unique critical point $c_0=c(f)$ which is nondegenerate and lies in $(0,1)$;
\item 
$f$ is orientation-preserving to the left of $c_0$ and orientation-reversing to the right of $c_0$;
\item 
$f(\del [0,1])\subset [0,1]$;
\item 
there is a unique expanding fixed-point in the interior of $[0,1]$;
\item 
the critical point is bounded from the critical value by a uniform constant.\footnote{If this bound is sufficiently large a neighbourhood of the renormalisation fixed-point will be contained in this space.} 
\end{enumerate}
Given $\Omega\subset \CC$, an open topological disk containing $[0,1]$, 
denote by 
$\U^\omega_\Omega$ the space of maps $f\in \End^\omega([0,1])$ 
satisfying properties (i)--(v) and the additional property:
\begin{enumerate}
\item[(vi)] $f$ admits a holomorphic extension to the domain $\Omega$, upon which it can be factored as $\psi\circ Q\circ \ii$ where $\ii\colon [0,1]\to [-a,1]$ is the unique orientation-preserving affine bijection between those domains, $Q\colon\CC\to\CC$ is given by $Q(z)=1-z^2$ and
$\psi\colon Q\circ\ii(\Omega)\to\CC$ is univalent and fixes the real axis;
\end{enumerate}
%
%
\begin{defn}[renormalisable]
A map $f$ in $\U^r$ (respectively, $\U_{\Omega}^\omega$) is \emph{renormalisable} if 
\begin{enumerate}
\item 
there is a subinterval $J^0\subset [0,1]$ containing the critical point such that $f^2(J^0)\subset J^0$;
\item 
the interiors of the subintervals $J^0$ and $J^1=f(J^0)$ are pairwise disjoint;
\item 
the map 
\begin{equation}
\R f=h^{-1}\circ f^2 \circ h
\end{equation}
is an element of $\U^r$ (respectively $\U_{\Omega}^\omega$) for some affine bijection $h$ from $[0,1]$ to $J^0$. 
Note there are exactly two such affine bijections, but there will only be one such that $\R f$ is in $\U^r$ (respectively $\U_{\Omega}^\omega$);
\end{enumerate}
The unimodal map $\R f$ is called the \emph{renormalisation of $f$} and the operator $\R$ is called the \emph{renormalisation operator}. 
\end{defn}
Let $\U^r_0$ and $\U_{\Omega,0}^\omega$ denote the respective subspaces of $\U^r$ and $\U_{\Omega}^\omega$ consisting of maps which are renormalisable. 
Recall that the operator $\R$ possesses a unique fixed-point, which we denote by $f_*$. This unimodal map is analytic and lies in $\U_{\Omega,0}^\omega$ for an appropriately chosen $\Omega$.
Moreover, $f_*$ is a hyperbolic fixed-point of the operator $\R$ with codimension-one stable manifold, and one-dimensional stable manifold.
\subsection{H\'enon-like Maps}\label{subsect:henonlike}
%
Let $\bar\e\in [0,1)$.
Denote by $\H^r(\bar\e)$ the space of $C^r$-diffeomorphisms onto their images $F\colon [0,1]^2\to \RR^2$ expressible as
\begin{equation}\label{eq:henonlike}
F(x,y)=\left(f(x)-\e(x,y),x\right)
\end{equation}
where $f\in\U^r$ and $\e\in C^r\left([0,1]^2,\RR_{\geq 0}\right)$ satisfies
\begin{enumerate}
\item 
$\e(x,0)=0$
\item 
$|\e|_{C^r, [0,1]^2}\leq \bar\e$
\end{enumerate}
Given $\Omega=\Omega_x\times \Omega_y\subset \CC^2$, an open topological bidisk containing $[0,1]^2$, 
denote by $\H_{\Omega}^{\omega}(\bar\e)$ the space of analytic diffeomorphisms onto their images 
$F\colon [0,1]^2\to\RR^2$ 
admitting a holomorphic extension to $\Omega$, which are expressible in the 
form~\eqref{eq:henonlike} where $f\in \U_{\Omegax}^\omega$ and $\e\in C^\omega\left([0,1]^2,\RR_{\geq 0}\right)$ satisfies property (i) above together with
\begin{enumerate}
\item[(ii)$_{\omega}$] 
$\e$ admits a holomorphic extension to $\Omega$ on which $|\e|_{\Omega}\leq \bar\e$, where $|\!-\!|_{\Omega}$ denotes the sup-norm on $\Omega$. 
\end{enumerate}
We call the map $\e$ a \emph{thickening} or a \emph{$\bar\e$-thickening} if we want to emphasise its \emph{thickness} $\bar\e>0$.
Denote by $\H^r$ the union of all $\H^r(\bar\e)$.
Define $\H^\omega_{\Omega}$ similarly. 
Denote by $\H^r(0)$ the subspace of the boundary of $\H^r$ consisting of maps whose thickening is identically zero. 
Again, define $\H^\omega_{\Omega}(0)$ similarly.
We call such maps \emph{degenerate H\'enon-like maps}.

Observe that the unimodal renormalisation operator 
$\R$ on $\U^r$ induces an operator, which we also 
denote by $\R$, on a subspace $\H_{0}^r(0)$ of $\H^r(0)$.
Similarly the renormalisation $\R$ acting on 
$\U^\omega_{\Omega_x}$ 
induces an operator on a subspace 
$\H^\omega_{\Omega, 0}(0)$ of $\H^\omega_{\Omega}(0)$.
In the analytic setting, a dynamical extension of this 
operator was constructed in~\cite[Section 3.5]{dCLM1}.  
More precisely it was shown that, for some (small) 
$\bar\epsilon>0$, and some choice of bidisk 
$\Omega\subset\CC^2$, 
there exists a non-trivial open subspace 
$\H^\omega_{\Omega,0}(\bar\epsilon)\subset \H^\omega_\Omega(\bar\epsilon)$
of maps $F$ satisfying the following property: 
there exists 
a unique non-flip saddle fixed-point $p_0$ 
(necessarily lying on the boundary of $[0,1]^2$) and 
a unique flip saddle fixed-point $p_1$ 
(lying in the interior of $[0,1]^2$).
Additionally, there exists an open subset $B^0\subset [0,1]^2$, containing $p_1$ in its boundary,
with the property that 
$B^1=F(B^0)$ and $B^0$ are disjoint and 
$F^2(B^0)\cap B^0$ is nonempty and connected.
For maps $F$ in $\H^\omega_{\Omega,0}(\bar\epsilon)$ a coordinate change $\MT\colon [0,1]^2\to B^0$, depending upon $F$, was constructed so that
\begin{equation}
\RH F=\MT^{-1}\circ F^2\circ \MT
\end{equation}
is again in $\H^\omega_\Omega(\bar\epsilon)$. 
However, unlike in the unimodal case, the map $\MT=\MT(F)$ is a {\it non-affine} diffeomorphism onto its image.
The coordinate change $\MT$ is canonical and chosen so that the operator 
$\RH\colon \H^\omega_{\Omega,0}(\bar\epsilon)\to\H^\omega_{\Omega}(\bar\epsilon)$
is an extension of the unimodal renormalisation operator on the space of renormalisable degenerate H\'enon-like maps $\H^\omega_\Omega(0)$. 
This dynamical extension is called the \emph{H\'enon renormalisation operator}, or simply the \emph{renormalisation operator}, on $\H_{\Omega,0}^{\omega}(\bar\e)$.
%

Clearly the map $F_*(x,y)=(f_*(x),x)$ is a fixed-point of this operator.
It was shown in~\cite[Section 4]{dCLM1} that this fixed-point is a hyperbolic fixed-point for the renormalisation operator and, moreover, the stable manifold has codimension-one.
%
%

Given a renormalisable map $F$, consider the non-affine coordinate change $\MT\colon [0,1]^2\to [0,1]^2$ in more detail. 
It is called the \emph{scope map} of $F$.
In~\cite[Section 1]{LM2} it was shown that the scope map $\MT$ can be extended to a vertical strip $A=[0,1]\times I$ containing $[0,1]^2$, so that it remains a diffeomorphism onto its image, and so that the image is a vertical strip contained in $[0,1]^2$ going from the top boundary segment of $[0,1]^2$ to the bottom boundary segment.

If we set $\oo{0}{\MT}=\MT$ and $\oo{1}{\MT}=F\circ\MT$ then $\oo{w}{\MT}$, for $w\in\mathcal \{0,1\}$, will be called the \emph{$w$-th scope map}. 

Let $\I^r(\bar\e)\subset\H^r(\bar\e)$ denote the subspace of infinitely renormalisable H\'enon-like maps.
Denote by $\I^r$ the union of all $\I^r(\bar\epsilon)$.
Then $\I^r\subset \H^r$.
Define $\I_{\Omega}^{\omega}(\bar\e)\subset \H_{\Omega}^{\omega}(\bar\e)$ and $\I^\omega_\Omega\subset H^\omega_\Omega$ similarly.
Given an infinitely renormalisable $F$, either in $\I^r$ or $\I^\omega_\Omega$, we will denote the $n$-th renormalisation $\RH^n F$ by $F_n$. 
For $w\in \{0,1\}$, let 
$\oo{w}{\MT_n}=F_n^w\circ\MT(F_n)\colon \Dom(F_{n+1})\to \Dom(F_{n})$ 
be the $w$-th scope map of $F_n$ as defined above, 
where $\Dom(F_n)$ denotes the domain of $F_n$. 
Then for $\word{w}{}=w_0\ldots w_n\in \{0,1\}^{n+1}$ the map
\begin{equation}
\MT^{\word{w}{}}
=
\MT_0^{w_0}\circ\cdots\circ \MT_n^{w_n}\colon \Dom(F_{n+1})\to \Dom(F_0)
\end{equation}
is called the \emph{$\word{w}{}$-scope map}. 
We denote the collection of all such functions by $\uline\MT$. 
That is, $\uline\MT=\{\MT^{\word{w}{}}\}_{\word{w}{}\in \{0,1\}^*}$ where, given an alphabet $\mathcal{W}$, we denote by $\mathcal{W}^*$ 
the set of all words of finite length over the alphabet $\mathcal{W}$.
With this in mind we define the \emph{renormalisation Cantor set} associated to $F$ by
\begin{equation}
\Cantor=\bigcap_{n\geq 0}\bigcup_{\word{w}{}\in \{0,1\}^n}\MT^{\word{w}{}}([0,1]^2).
\end{equation}
That this is a Cantor set was shown in~\cite{dCLM1}. 
For a point $z\in\Cantor$ the corresponding word $\word{w}{}$ is called the \emph{address} of $z$. 
In particular we define the \emph{tip} $\tau=\tau(F)$ to be the point in $\Cantor$ with address $\word{w}{}=0^\infty$. 
In other words
\begin{equation}
\tau=\bigcap_{n\geq 1}\MT^{0^n}([0,1]^2)
\end{equation}
This is the point which in~\cite{dCLM1} replaced the role of the critical value in the renormalisation theory for unimodal maps. 
We remark that in~\cite{LM1} it was shown that 
$W(\tau)=\bigcap_{n\geq 1} \MT^{\word{0^n}{}}(A_n)$ 
coincides with the stable manifold of $\tau$, where $A_n=[0,1]\times I_n$ is the vertical strip which is the domain of the extended scope map $\MT^0_n$.

The action of $F$ on $\Cantor$ is metrically isomorphic to the adding machine.
Hence $\Cantor$ has a unique $F$-invariant measure, $\mu$.
The \emph{Average Jacobian} $b=b(F)$ is then defined by
\begin{equation}\label{eq:averagejacobian} 
b(F)=\exp\int \log|\jac F|d\mu.
\end{equation}
Now we can state the main result of~\cite{dCLM1}, which we shall refer to as the \emph{asymptotic formula}.
\begin{thm}\label{thm:universality}
Given $F\in\I_{\Omega}(\bar\e_0)$ there exists a universal $a\in C^\omega([0,1],\RR)$ and universal $0<\rho<1$, depending upon $\Omega$ only, such that
\begin{equation}\label{eq:asymptotic-formula}
F_n(x,y)=(f_n(x)-b^{2^n}a(x)y(1+\bigo(\rho^n)),x)
\end{equation}
where $f_n$ are unimodal maps converging exponentially to $f_*$, the unimodal fixed-point of renormalisation.
\end{thm}
In~\cite{LM1} it was shown that there exists a dynamically-defined presentation of $\Cantor$, which we call the \emph{topological boxing} and denote by $\uline{D}$. 
That is $\uline D=\{D^{\word{w}{}}\}_{\word{w}{}\in \{0,1\}^*}$ is a collection of topological disks satisfying 
\begin{enumerate}
\item $F(D^{\word{w}{}})\subset D^{1+\word{w}{}}$ for all $\word{w}{}\in \{0,1\}^*$,
\item $D^{\word{w}{}}$ and $D^{\word{w}{}'}$ are disjoint for all $\word{w}{}\neq\word{w}{}'$ of the same length,
\item the disjoint union of the $D^{\word{w}{w}}$, $w\in \{0,1\}$, is a subset of $D^{\word{w}{}}$, for all $\word{w}{}\in \{0,1\}^*$,
\item $\Cantor^{\word{w}{}}\subset D^{\word{w}{}}$ for all $\word{w}{}\in \{0,1\}^*$,
\end{enumerate}
The definition of these topological disks is as follows. 
The reader can compare with figure~\ref{fig:extendedWs}.
(We have changed notation slightly from~\cite{dCLM1} 
and~\cite{LM1} for simplicity.)
Let $F$ be a H\'enon-like map in $\H^r$ or $\H^\omega_\Omega$ 
and with a unique non-flip saddle fixed-point $p_0$ and 
unique flip saddle fixed-point $p_1$.
We say that $F$ is \emph{(topologically) renormalisable} 
if $W^u(p_0)$ intersects $W^s(p_1)$ in a single orbit 
$\{r^i\}_{i\in\ZZ}$ (we will fix the indexing in a moment), 
and this intersection is transverse.

For $|\e |_\Omega$ sufficiently small we may assume that 
$W^s_{\mathrm{loc}}(p_1)$ separates $[0,1]^2$ into exactly 
two connected components. (This is certainly the case for 
a degenerate H\'enon-like map; such maps have all local stable 
manifolds being vertical straight lines from the top boundary 
arc of $[0,1]^2$ to the bottom boundary arc of $[0,1]^2$.) 
It follows that there is a first 
intersection point, as we travel from $p_0$ along $W^u(p_0)$, 
between $W^s_{\mathrm{loc}}(p_1)$ and $W^u(p_0)$. We 
denote this point by $r^0$ and define $r^i=F^i(r^0)$ for all 
$i\in\ZZ$.

Observe that the curves $[r^0,r^1]^s$ and $[r^0,r^1]^u$ bound a region $D^0$ with the property that $F^2(D^0)\subset D^0$. If we let $D^1=F(D^0)$ then $\{D^0,D^1\}$ form the first level of the \emph{topological boxing} of $F$. 

In the case when $F$ is infinitely renormalisable the same argument can be applied to the $n$-th renormalisation $F_n$, for each positive integer $n$. 
This time, compare with figure~\ref{fig:boxings}.
Namely, there exists a non-flip saddle fixed-point $p_{0,n}$ and a flip saddle fixed-point $p_{1,n}$ such that $W^u(p_{0,n})$ and $W^s(p_{1,n})$ have intersection equal to a single orbit $\{r^i_n\}_{i\in\ZZ}$ and this intersection is transverse.
Here the indexing is chosen so that, again, the first intersection point, travelling from $p_{0,n}$ along $W^u(p_{0,n})$, between $W^u(p_{0,n})$ and $W^s(p_{1,n})$ is $r_n^0$.
Then the curves $[r_n^0,r_n^1]^s$ and $[r_n^0,r_n^1]^u$ again bound a region $D^0_n$ with the property that $F_n^2(D^0_n)\subset D_n^0$. 
As before we set $D^1_n=F_n(D^0_n)$.
\begin{landscape}
\begin{figure}[htp]
\centering
\small
\psfrag{p00}{$p_{0,0}$}
\psfrag{p10}{$p_{1,0}$}
\psfrag{r00}{$r_{0}^0$}
\psfrag{r10}{$r_{0}^1$}
\psfrag{r20}{$r_{0}^2$}
\psfrag{p20}{$p_{2,0}$}
\psfrag{D00}{$D^0_{0}$}
\psfrag{D10}{$D^1_{0}$}
\psfrag{Wup00}{$W^u(p_{0,0})$}
\psfrag{Wsp10}{$W^s(p_{1,0})$}
\psfrag{Wup10}{$W^u(p_{1,0})$}
\psfrag{Wsp20}{$W^s(p_{2,0})$}
\psfrag{MT}{$\Psi$}
\psfrag{p01}{$p_{0,1}$}
\psfrag{p11}{$p_{1,1}$}
\psfrag{r01}{$r_{1}^0$}
\psfrag{r11}{$r_{1}^1$}
\psfrag{r21}{$r_{1}^2$}
\psfrag{D01}{$D^0_{1}$}
\psfrag{D11}{$D^1_{1}$}
\psfrag{Wup01}{$W^u(p_{0,1})$}
\psfrag{Wsp11}{$W^s(p_{1,1})$}
\includegraphics[width=1.5\textheight]{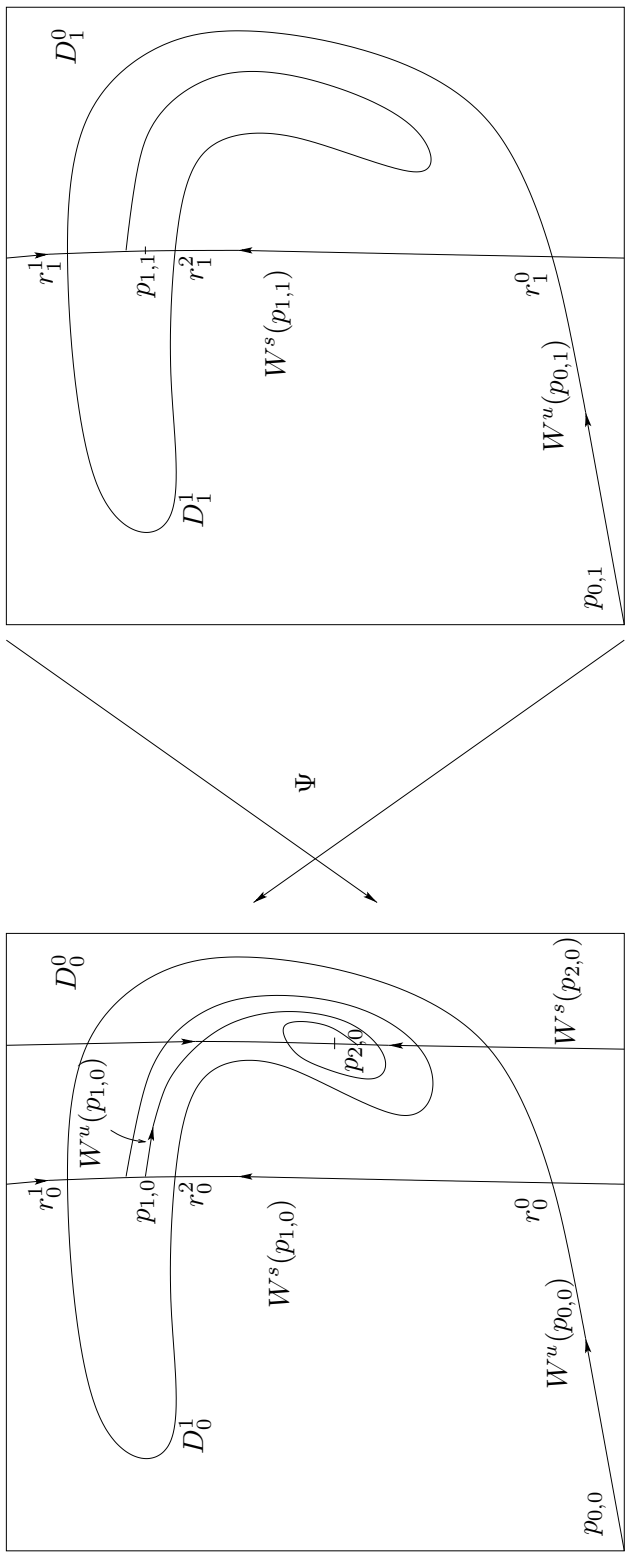}
\caption{The topological boxing}\label{fig:boxings}
\end{figure}
\end{landscape}
\noindent
Applying the scope maps then gives us the complete topological boxing.
Namely, given $\word{w}{}\in \{0,1\}^n$ and $w\in \{0,1\}$ we define $D^{\word{w}{}w}=\MT^{\word{w}{}}(D^w)$.
We then define the topological boxing by $\uline{D}=\{D^{\word{w}{}}\}_{\word{w}{}\in \{0,1\}^*}$.

\subsection{The Average Jacobian as a Topological Invariant}\label{sect:kappa}
Consider the following construction from~\cite{LM1}.
We observe that the construction also works in the $C^r$-category.
However, for expositional simplicity we restrict ourselves to the analytic case.
Let $F\in\mathcal H^\omega_{\Omega,0}(\bar\epsilon)$.
Define 
\begin{equation}
\mathcal M=[p_1,r^0]^s\cup F^{-1}\left([p_1,r^0]^s\right)\cup F^{-2}\left([p_1,r^0]^s\right)
\end{equation}
Then, for $\bar\epsilon$ sufficiently small, $\mathcal M$ consists of four connected components $M_i, i=-2,-1,0,1$ indexed so that
\begin{itemize}
\item $M_0=W^s_{\mathrm{loc}}(p_1)$
\item $M_1\cap W^u(p_0)=\emptyset$
\item $M_{i}\cap W^u(p_0)=\{r^i\}$ for $i=-1,-2$
\end{itemize}
We refer the reader to figure~\ref{fig:extendedWs}.
If $F$ is $n$-times renormalisable we can similarly define
\begin{equation}
\mathcal M_n=[p_{1,n},r^0_n]^s\cup F_n^{-1}\left([p_{1,n},r^0_n]^s\right)\cup F_n^{-2}\left([p_{1,n},r^0_n]^s\right)
\end{equation}
$\mathcal M_n$ consists of four connected components $M_{i,n}, i=-2,-1,0,1$ indexed so that
\begin{itemize}
\item $M_{0,n}=W^s_{\mathrm{loc}}(p_{1,n})$
\item $M_{1,n}\cap W^u(p_{0,n})=\emptyset$
\item $M_{i,n}\cap W^u(p_{0,n})=\{r^i_n\}$ for $i=-1,-2$
\end{itemize}
Finally, let $M_i^{0^n}=\MT^{0^n}(M_{i,n})$ for $i=-2,-1,0,1$ and $n\in\NN$.
\begin{figure}[htp]
\centering
\small
\psfrag{p00}{$p_{0,0}$}
\psfrag{p10}{$p_{1,0}$}
\psfrag{r00}{$r_{0}^0$}
\psfrag{r10}{$r_{0}^1$}
\psfrag{r20}{$r_{0}^2$}
\psfrag{p20}{$p_{2,0}$}
\psfrag{D00}{$D^0_{0}$}
\psfrag{D10}{$D^1_{0}$}
\psfrag{Wup00}{$W^u(p_{0,0})$}
\psfrag{Wsp10}{$W^s(p_{1,0})$}
\psfrag{Wup10}{$W^u(p_{1,0})$}
\psfrag{Wsp20}{$W^s(p_{2,0})$}
\psfrag{MT}{$\Psi$}
\psfrag{p01}{$p_{0,1}$}
\psfrag{p11}{$p_{1,1}$}
\psfrag{r01}{$r_{1}^0$}
\psfrag{r11}{$r_{1}^1$}
\psfrag{r21}{$r_{1}^2$}
\psfrag{D01}{$D^0_{1}$}
\psfrag{D11}{$D^1_{1}$}
\psfrag{Wup01}{$W^u(p_{0,1})$}
\psfrag{Wsp11}{$W^s(p_{1,1})$}

\psfrag{M_-2}{$M_{-2}$}
\psfrag{M_-1}{$M_{-1}$}
\psfrag{M_0=W^s}{$M_{0}=W^s_{\mathrm{loc}}(p_1)$}
\psfrag{M_1}{$M_{1}$}
\psfrag{p_0}{$p_{0}$}
\psfrag{p_1}{$p_{1}$}
\psfrag{W^u}{$W^u(p_0)$}
\psfrag{r^-2}{$r^{-2}$}
\psfrag{r^-1}{$r^{-1}$}
\psfrag{r^0}{$r^{0}$}
\psfrag{r^1}{$r^{1}$}
\psfrag{r^2}{$r^{2}$}
\psfrag{D^0}{$D^{0}$}
\psfrag{D^1}{$D^{1}$}

\includegraphics[width=0.5\textheight]{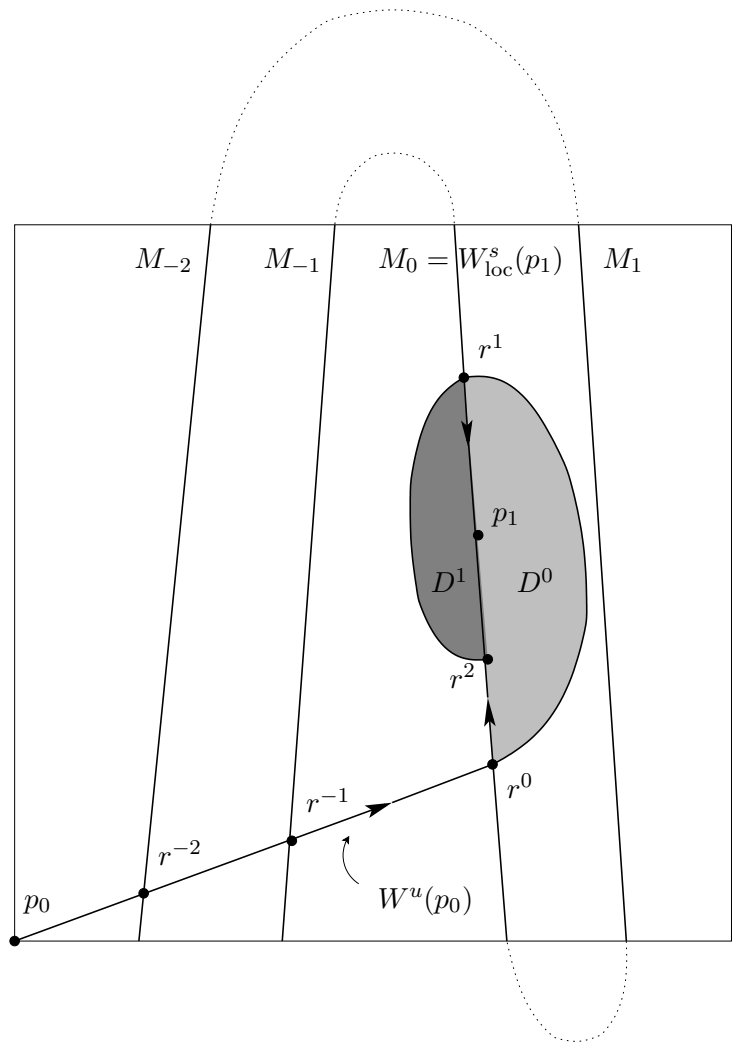}
\caption{The extended stable manifold}\label{fig:extendedWs}
\end{figure}
%
We wish to know how $M_i^{0^n}$ accumulates upon the tip $\tau$ of $F$.
For $|\e |_\Omega$ sufficiently small observe that $W^s_{\mathrm{loc}}(\tau)$ separates $[0,1]^2$ into exactly two connected components. 
It also separates $D^0$ into two connected components.
Let $D(\tau)$ denote the connected component of $D^0\setminus W^s_{\mathrm{loc}}(\tau)$ not containing $p_1$.
Define
\begin{equation}
\kappa_F=\min\left\{k\in\NN : D(\tau)\cap M_1^{0^k}\neq \emptyset\right\}
\end{equation}
and
\begin{equation}\label{def:kappa}
\boldsymbol{\kappa} _F=\lim_{n\to\infty}\frac{\kappa_{\RH^n F}}{2^n}
\end{equation}
whenever the limit exists.
Then the following was shown in~\cite{LM1}.
\begin{thm}[Lyubich--Martens]\label{thm:b-topinv}
There exists $\bar\epsilon>0$ and a bidisk $\Omega$ containing $[0,1]^2$ such that the following holds:
Let $F\in \I^\omega_\Omega(\bar\epsilon)$.
Then
\begin{equation}
\boldsymbol{\kappa}_F=\frac{1}{2}\frac{\log b}{\log \sigma}
\end{equation}
where $b=b_F>0$ denotes the average Jacobian of $F$ and $\sigma$ denotes the one-dimensional period-doubling scaling ratio.
\end{thm}
\begin{rmk}
This shows in particular, since $\kappa_{\R^n F}$ is a topological invariant of $F$ for each $n>0$, that the average Jacobian is a topological invariant.
\end{rmk}

Recall the following definition from~\cite{LM1}. 
Given $F\in\I^r$ we say that $F$ possesses an \emph{$(m,n)$-heteroclinic tangency} if there exists a point of tangency $q_{m,n}$ between $W^u(p_m)$ and $W^s(p_n)$, where $p_m$ and $p_n$ denote the periodic orbits of periods $2^m$ and $2^n$ respectively.

\begin{lem}\label{lem:het-tangency}
Given $[b_{\mathrm{min}},b_{\mathrm{max}}]\subset [0,\bar b)$ let $F\in C^1([b_{\mathrm{min}},b_{\mathrm{max}}],\mathcal I^\omega_\Omega)$ be a one-parameter family parametrised by the average Jacobian.
For any positive integer $N$ there exist integers $m$ and $n$ satisfying  $N<m<n$ and a parameter $b\in [b_{\mathrm{min}},b_{\mathrm{max}}]$ such that $F_b$ possesses an $(m,n)$-heteroclinic tangency $q_{m,n}\in W^u(p_m)\cap W^s(p_n)$.
\end{lem}
\begin{proof}
Given a family $F_b$ as above, we denote by $F_{\boldkappa}$ the reparametrisation of the family by the invariant $\boldkappa$. By Theorem~\ref{thm:b-topinv} such a smooth reparametrisation exists and, moreover, satisfies $\boldkappa_{F_\kappa}=\boldkappa$.

Given an interval $[\boldkappa_{\mathrm{min}},\boldkappa_{\mathrm{max}}]$, denote its length by $l$. 
By definition~\ref{def:kappa} and Theorem~\ref{thm:b-topinv} (which shows the limit in the definition of $\boldkappa$ exists) whenever a positive integer $m$ is sufficiently large
\begin{equation}
\left| \boldkappa-\frac{\kappa_{\RH^m F_{\boldkappa}}}{2^m}\right| < \frac{l}{3}
\end{equation}
for all $\boldkappa\in [\boldkappa_{\mathrm{min}},\boldkappa_{\mathrm{max}}]$. 
Consequently there exists an positive integer $N>0$ such that
\begin{equation}
\left|\frac{\kappa_{\RH^m F_{\boldkappa_{\mathrm{max}}}}}{2^m}-\frac{\kappa_{\RH^m F_{\boldkappa_{\mathrm{min}}}}}{2^m}\right|>l-\frac{2l}{3}=\frac{l}{3}
\end{equation}
whenever $m>N$.
By increasing $N$ if necessary, we may further assume that $2^m l/3>1$ for all $m>N$. 
Then it follows that 
\begin{equation}
\left|\kappa_{\RH^m F_{\boldkappa_{\mathrm{max}}}}-\kappa_{\RH^m F_{\boldkappa_{\mathrm{min}}}}\right|>1
\end{equation}
for all $m>N$. 
By continuity this implies that for any $m>N$ there exists a parameter $\boldkappa\in [\boldkappa_{\mathrm{min}},\boldkappa_{\mathrm{max}}]$ such that, for the corresponding map $\RH^mF_{\boldkappa}$, the arc $\del^u D^m(\tau_m)\subset W^u(p_{0,m})$ is tangent to the curve $M_{1,m}^{0^{n}}\subset W^s(p_{m+n,m})$ (where for simplicity we write $n=\kappa_{\RH^mF_{\boldkappa}}$). 
Denote this point of tangency by $q_{0,m+n; m}$.

By taking the diffeomorphic image of these objects under the scope map $\MT^{0^m}$, it follows that the diffeomorphism $F_{\boldkappa}$ possesses a heteroclinic tangency $q_{m, n; 0}=\MT^{0^m}(q_{0, m+n; m})$ between the saddles $p_{m, 0}$ and $p_{n, 0}$. 
\end{proof}

\begin{rmk}\label{rmk:arcs-disjoint-from-pieces}
It follows from the argument above that the unstable arc $[p_{m,0},q_{m,n;0}]^u=\MT^{0^m}([p_{0,m},q_{0,m+n; m}]^u)$ does not intersect any of the pieces of depth $n+2$ or greater.

It also follows that the stable arc $[q_{m, n; 0},p_{n, 0}]^s$, when intersected with $D^{0^m}$ consists of finitely many arcs passing from the top of $D^{0^m}$ to the bottom together with one arc from $q_{m,n}$ to the boundary and one arc from $p_{n, 0}$ to the boundary.

These arcs are disjoint from any piece of depth $n+2$ and therefore the whole arc $[q_{m, n; 0},p_{n, 0}]^s$ does not intersect any piece of depth $n+2$ or greater.
\end{rmk}

\subsection{Saddle Connections}\label{subsect:Palis}
Next we consider another invariant of topological conjugacy discovered by Palis~\cite{Palis}.
Let $F$ be an orientation-preserving diffeomorphism of the plane (or a surface) for which there exists
\begin{enumerate}
\item fixed saddles $p_0$ and $p_1$ (not necessarily distinct)
\item a point $q\in W^u(p_0)\cap W^s(p_1)$ such that $W^u(p_0)\not\pitchfork_q W^s(p_1)$
\end{enumerate}
Let $\lambda^{s}_{j}$ and $\lambda^{u}_{j}$ denote respectively the contracting and expanding eigenvalues of $DF$ at $p_j$, where $j=0$ or $1$. 
Then Palis showed the following (see also~\cite{Guckenheimer} for the proof).
\begin{thm}[Palis~\cite{Palis}]\label{thm:Palis}
Given $F$, $p_0$ and $p_1$ as above the quantity 
\begin{equation}\label{eqn:palisinv}
P_{F:p_0,p_1}=\frac{\log |\lambda_0^s|}{\log |\lambda_1^u|}
\end{equation}
is a topological invariant. 
Consequently, if the orientation-preserving diffeomorphism $F'$ is topologically conjugate to $F$ 
with corresponding saddles $p_i'$, then $P_{F':p_0',p_1'}=P_{F:p_0,p_1}$.
\end{thm}
\begin{rmk}
The above can be generalised to periodic saddles by taking a sufficiently large iterate of $F$.
Then the above ratio would have the factor $n_1/n_0$ in front, where $n_i$ is the period of $p_i$.
However, since the periods are also topologically invariant data we can forget this factor and consider $P_{F:p_0,p_1}$ defined as above in this case too.
\end{rmk}
%
%
Let $F\in\H^r$ have periodic saddle points $p_0$ and $p_1$ as above.
As hyperbolicity is an open property, there exists a neighbourhood $\mathcal P=\mathcal P_{F:p_0,p_1}$ of $F$ in $\H^r$ such that the saddles $p_0$ and $p_1$ persist. 
Therefore $P_{F:p_0,p_1}$ extends to a well-defined map 
\begin{equation}\label{eq:palismap}
P_{F:p_0,p_1}\colon \mathcal P_{F:p_0,p_1}\subset \H^r\to \RR
\end{equation}
Note that this map is not a topological invariant on all of $\mathcal P_{F:p_0,p_1}$.
However, there exists a subspace $\mathcal Q_{F:p_0,p_1,q}\subset \mathcal P_{F:p_0,p_1}$ containing diffeomorphisms $F'$ such that, the continuations $p_0'$ and $p_1'$ of the saddles $p_0$ and $ p_1$ respectively possess a tangency $q'$ between the relevant stable and unstable manifolds.
Hence the restriction $P_{F: p_0,p_1}\colon\mathcal Q_{F,p_0,p_1}\to\RR^d$ is a topological invariant and will be called the \emph{Palis map with markings} $F,p_0,p_1$.
\begin{prop}
For any choice of $F\in\H^r$ with a saddle connection as above there exists a neighbourhood $\mathcal P\subset \H^r$ containing $F$ such that
$P_{F: p_0,p_1}\in C^1(\mathcal P,\RR^d)$.
\end{prop}
\begin{proof}
This follows immediately from the definition~\eqref{eq:palismap} 
above and the fact that derivatives of hyperbolic fixed-points 
have eigenvalues which vary smoothly.
\end{proof}
\noindent
Fix an open neighbourhood $W$ of $q$ not containing $p_0$ or $p_1$.
Let
\begin{itemize}
\item 
$l$ denote the connected component of $W^s(p_1)\cap W$ containing $q$
\item 
$\bar{l}$ denote the arc of minimal length in $W^s(p_1)$ containing $p_1$ and $l$
\item 
$k$ denote the connected component of $W^u(p_0)\cap W$ containing $q$
\item 
$\bar{k}$ denote the arc of minimal length in $W^u(p_0)$ containing $p_0$ and $k$
\end{itemize}
\begin{defn}
A coordinate change $\beta\colon (W,q)\to (\RR^2,0)$ is \emph{horizontal} if for some $a\geq 2$,
\begin{itemize}
\item 
$\beta$ preserves horizontal lines
\item 
$\beta(l)\subset \{x=0\}$
\item 
$\beta(k)\subset \{x=-|y|^a\}$
\end{itemize}
The coordinate change is \emph{vertical} if for some $a\geq 2$,
\begin{itemize}
\item 
$\beta$ preserves vertical lines
\item 
$\beta(k)\subset \{y=0\}$
\item 
$\beta(l)\subset \{y=-|x|^a\}$
\end{itemize}
\end{defn}
\begin{rmk}\label{rmk:horiz-vert-coords}
The $(m,n)$-heteroclinic tangency $q$ constructed in 
Section~\ref{sect:kappa} has the property that in a 
neighbourhood of $q$, the components of $W^u(p_m)$ and 
$W^s(p_n)$ passing though $q$ are graphs over the 
$y$-axis, and that their preimages are graphs over the 
$x$-axis. 
From this it can be shown that there exists a horizontal 
change of coordinates at $q$ and a vertical change of 
coordinates as above.
\end{rmk}
Let $\mathcal W\subset \mathcal P_{F,p_0,p_1}$ denote 
the set of H\'enon-like diffeomorphisms $F'$ such that 
$p_0$ is also a fixed-point of $F'$, and such that 
$\bar{l}\subset W^u(p_0;F')$.
\begin{rmk}
If we did not restrict our attention to $\mathcal W$ 
we would construct, for $F'\in\mathcal P_{F,p_0,p_1}$, 
a coordinate change $\beta_{F'}$ so that 
$\beta_{F'}(l_{F'})=\{x=0\}$.
However, as the perturbations we will need to make 
later will leave $F$ unchanged in a neighbourhood 
of $\bar{l}$ we will consider this simplified case only.
\end{rmk}
\begin{prop}
Let $F$, $p_0$, $p_1$, and $q$ be as above.
There exists $\epsilon>0$ such that any 
$C^2$-$\epsilon$-small perturbation 
$F'\in\mathcal W$ of $F$ satisfies
\begin{equation}
\beta(l_{F'})\subset \{x=0\},\qquad \beta(k_{F'})\subset\{x=\psi(y)\}
\end{equation}
for some unimodal map 
$\psi=\psi_{F'}\in C^2(\RR,\RR)$ 
depending upon $F'$.
\end{prop}
\begin{proof}
Finite segments of invariant manifolds of hyperbolic 
fixed points vary smoothly under smooth perturbation.
Hence given $\epsilon_1>0$, there exists $\epsilon>0$ 
such that for any $C^2$-$\epsilon$-small perturbation $F'$ 
of $F$, the corresponding arc $\bar{k}_{F'}$ is 
$C^2$-$\epsilon_1$-close to $\bar{k}_{F}$.
Moreover, if $F'$ lies in $\mathcal{W}$, $\bar{l}_{F'}$ 
coincides with $\bar{l}_{F}$.
Consequently, for $\epsilon$ sufficiently small, 
$\beta(\bar{k}_{F'})$ is the graph of a $C^2$-smooth 
function $\psi$ which is $C^2$-$\epsilon_1$-close to 
$-|x|^a$ and $\beta(\bar{l}_{F'})\subset \{x=0\}$.

Finally, since $\psi'$ is $C^1$-$\epsilon_1$-close to 
the derivative of $-|x|^a$, it also has a unique zero 
if $\epsilon_1$ sufficiently small.
Hence  $\psi$ is also unimodal if $\epsilon$ is 
sufficently small.
\end{proof}
Given a marking $F$, $p_0$, $p_1$, and $q$ and given the $\epsilon>0$ determined by the above proposition,
we denote by $\mathcal W_0$ the $C^2$-$\epsilon$-neighbourhood of $F$ in $\mathcal W$.
For $F'\in\mathcal W_0$ let $c_{F'}$ denote the unique critical point of $\psi_{F'}$. 
\begin{defn}
Let $F$, $p_0$, $p_1$ and $q$ be as above.
For $F'\in\mathcal V$, define
\begin{equation}
Q_{F,p_0,p_1,q}(F')=\psi_{F'}(c_{F'})
\end{equation}
\end{defn}

\section{The Main Construction}\label{sect:mainconstr}
Now we come to the proof of the main theorem.
For each positive integer $d$ we will construct a $d$-parameter family of infinitely renormalisable H\'enon-like maps which are topologically distinct. 
We break this into three steps:
\begin{enumerate}
\item[(A)] 
First, we construct a $d$-parameter family of Henon-like maps such that at an initial parameter there are $d$ distinct $(m,n)$-heteroclinic tangencies. 
Moreover, at this parameter the stable multiplier of $p_m$, corresponding to the $(m,n)$-heteroclinic tangency, varies regularly with the parameter.

\item[(B)] 
Second, given the family from (A), we construct a $2d$-parameter family where each additional parameter controls a local perturbation in a neighbourhood of one of the $d$ points of tangency for the initial parameter in (A).

\item[(C)] 
Finally, given a family from (B), we show that those parameters with $d$ tangencies form locally a $d$-parameter submanifold. 
Restricting the $2d$-parameter family to this submanifold gives a $d$-parameter family with $d$-tangencies which persist.  
\end{enumerate}
\subsection{Construction of the First Family}\label{subsect:1stconstr}
Let $F$ be an infinitely renormalisable H\'enon-like map.
Given integers $M$ and $N$ satisfying $0\leq M<N$ define
\begin{equation}
T_{M,N}=
\bigcup_{\word{w}{}\in\mathcal \{0,1\}^{M}}D^{\word{w}{}}
\setminus 
\overline{\bigcup_{\word{w}{}\in\{0,1\}^{N}}D^{\word{w}{}}}
\quad
\mbox{and}
\quad
T_{M,*}=
\bigcup_{\word{w}{}\in\{0,1\}^{M}}D^{\word{w}{}}
\end{equation}
We adopt the convention that $\{0,1\}^0=\emptyset$ and that $D^\emptyset=[0,1]^2$.
It then follows that $T_{0,N}=[0,1]^2\setminus \bigcup_{\word{w}{0}\in\{0,1\}^N}D^{\word{w}{}}$ for each positive integer $N$.
The aim of this section is to show the following.
\begin{thm}[First Construction]\label{thm:constr1}
Let $r\in \{3,4,\ldots,\infty\}$.
For each integer $d\geq 2$ there exists 
\begin{itemize}
\item 
$\BBB\subset \RR^d$ an open subset, 
\item 
$F\in C^1(\BBB,\mathcal I^r)$,
\item 
$0=N_0<m_1<n_1<N_1<\ldots<N_{d-2}<m_{d-1}<n_{d-1}<N_{d-1}$, 
\item 
$b^*=(b^*_1,b^*_2,\ldots,b^*_d)\in \BBB$ 
\end{itemize}
such that
\begin{enumerate}
\item 
for $i\neq d$, $F(b)|T_{N_{i-1},N_i-1}(b)$ depends only on $b_i$,
\item 
for $i=d$, $F(b)|T_{N_{d-1},*}(b)$ depends only on $b_d$,
\item 
$b(F_{b})=b_d$, where $b(F)$ denotes the average Jacobian of $F$,
\item 
for $i=1,2\ldots,d-1$, at $b=b^*$:
\begin{enumerate}
\item[(a)] 
$F_{b}$ possesses an $(m_i,n_i)$-heteroclinic tangency $q_i$ such that $[p_{m_i},q_i]^u, [q_i,p_{n_i}]^s\subset T_{N_{i-1},N_{i}-1}$
\item[(b)] 
$\lambda^s_{m_i}$ varies regularly with $b_i$
\end{enumerate}
\end{enumerate} 
\end{thm}
The construction is by induction.
First we take a one-parameter family $F_{b_1}$ of infinitely renormalisable H\'enon-like maps parametrised by the average Jacobian. 
It can be shown by the discussion in Section~\ref{sect:kappa} that for some parameter there is an $(m_1,n_1)$-heteroclinic tangency for some $m_1<n_1$.
From this we construct, via a bump-function argument, a two-parameter family $F_{b_1,b_2}$, 
so that the support of the first parameter is contained in the union of sets of level $N_1$
and the second parameter is contained in the the complement of the sets of level $N_1-1$.
Moreover, for each parameter the new family coincides with the original family on its support.

We will now show that there exists a good one-parameter family.
First we need a preliminary lemma.
\begin{prop}\label{prop:a-nonzero}
Let $a$ denote the universal function from the asymptotic formula~\eqref{eq:asymptotic-formula}.
Let $f_*$ denote the unimodal period-doubling renormalisation fixed-point.
Let $p_*$ denote the $p_1$-fixed-point for $f_*$.
Then $a(p_*)\neq 0$.
\end{prop}
\begin{proof}
Recall that $a(x)=\frac{v_*'(x)}{v_*'(f_*(x))}$, where $v_*, f_*$ are universal, analytic, and non-constant.
Observe that if the numerator is zero at $x=p_*$ then so is the denominator.
Similarly it can be shown that if the $n$-th derivative of the numerator is zero at $x=p_*$ the $n$-th derivative of the denominator is also zero.
However, $v_*$ is analytic and non-constant. 
Therefore there exists a first $n$ such that $v_*^{(n+1)}(p_*)\neq 0$. 
It follows, by analyticity of $v_*'\circ f$ at $p_*$ together with l'Hopital's rule, that $a(p_*)\neq 0$.
\end{proof}
Now let us show that there exist good families.
\begin{prop}\label{prop:good-families}
Let $F\in C^1\left((0,\Bar{b}],\I^\omega_\Omega\right)$ be parametrised by the average Jacobian.
There exists $\Bar{\Bar{b}}\in (0,\Bar{b}]$ such that for all sufficiently large positive integers $n$, 
$\lambda_n^{s}(b)$ is a regular function of the parameter $b$ at all $b\in (0,\Bar{\Bar{b}}]$.
\end{prop}
\begin{proof}
Let $F_b$ be any one-parameter family of infinitely renormalisable H\'enon-like maps parametrised by the average Jacobian $b$.
Consider $\R^nF_b$.
Let $p_{1,n}(b)=\left(x_n(b),y_n(b)\right)$ denote the unique flip-saddle fixedpoint for $\R^nF_b$.
Let $\mathsf{T}_n(b)$ and $\mathsf{D}_n(b)$ denote respectively the trace and determinant of $\mathsf{M}_n(b)=D \R^nF_b(p_{1,n}(b))$. 
Let $\lambda_n^+(b)$ and $\lambda_n^-(b)$ denote the two eigenvalues of $\mathsf{M}_n(b)$, where $\pm$ is determined by which sign $\pm$ is used in the quadratic formula.
Then
\begin{equation}
\lambda_n^\pm(b)^2-\mathsf{T}_n(b)\lambda_n^\pm(b)+\mathsf{D}_n(b)
=
0
\end{equation}
Differentiating with respect to $b$ and rearranging we find
\begin{equation}
\del_b\lambda_n^\pm (2\lambda_n^\pm-\mathsf{T}_n)
=
\del_b\mathsf{T}_n\lambda_n^\pm-\del_b\mathsf{D}_n
\end{equation}
By convergence of renormalisation
\begin{equation}\label{eq:eigenvals-asymp}
\lambda_n^+= -b^{2^n}(1+\bigo(\rho^n)), \qquad 
\lambda_n^-= f_*'(p_{1,*})(1+\bigo(\rho^n))
\end{equation}
This implies that $|2\lambda_n^\pm-\mathsf{T}_n|=|\lambda_n^+-\lambda_n^-|$ is bounded from above for $n$ sufficiently large. 
Consequently, for sufficiently large $n$, it follows that $\del_b\lambda_n^\pm=0$ if and only if $\lambda_n^\pm=\del_b\mathsf{D}_n/\del_b\mathsf{T}_n$.

Hence, by equations~\eqref{eq:eigenvals-asymp}, to show that $\del_b\lambda_n^\pm\neq 0$  it suffices to show that $\del_b D_n/\del_bT_n$ is not of the order $-b^{2^n}$ or $1$.
%
A computation gives
\begin{align}
\del_b \mathsf{T}_n(b)&=
\del_b(\tr D\R^nF)_{b,p_{1,n}(b)}
+\del_{x,y}(\tr D\R^nF)_{b,p_{1,n}(b)}\cdot \del_b p_{1,n}(b) \\
\del_b \mathsf{D}_n(b)&=
\del_b(\det D\R^nF)_{b,p_{1,n}(b)} 
+\del_{x,y}(\det D\R^nF)_{b,p_{1,n}(b)}\cdot \del_b p_{1,n}(b)
\end{align}

\begin{claim}
Let $F\in C^1([0,\bar{b}),\I^\omega_\Omega)$ be parametrised by the average Jacobian.
There exists a positive integer $N$, $\Bar{\Bar{b}}>0$ and $C_0>0$ such that the following holds:
For $b\in [0,\Bar{\Bar{b}})$ and $n>N$,
\begin{enumerate}
\item 
$\left|\del_b (\det D\R^nF)_{b,p_{1,n}(b)}\right|>C_0^{-1}b^{2^n-1} 2^n$
\item 
$\left\|\del_{x,y} (\det D\R^nF)_{b,p_{1,n}(b)}\right\|<C_0b^{2^n}$
\item 
$\left|\del_b (\tr D\R^n F)_{b,p_{1,n}(b)}\right|<(3/2)^n$
\item 
$\left\|\del_{x,y} (\tr D\R^n F)_{b,p_{1,n}(b)}\right\|<C_0$
\end{enumerate}
\end{claim}
\noindent
\emph{Proof of Claim:}
By the asymptotic formula
\begin{equation}
\R^nF_b(x,y)=\left(f_{b,n}(x)-a(x) b^{2^n} y E_n(x,y),x\right)
\end{equation}
where $E_n(x,y)=1+\bigo(\rho^n)$.
It follows that
\begin{align}
\det D\R^n F_b(x,y)
&=b^{2^n}\del_y\left(a(x)E_n(x,y)\right) \\ 
\tr D\R^n F_b(x,y)
&=f'_{b,n}(x)-b^{2^n}y\del_x\left(a(x)E_n(x,y)\right)
\end{align}
\begin{enumerate}
\item
Observe that $\left.\del_y(aE_n)\right|_{p_{1,n}}=a(x_n)\del_yE_n(x_n,y_n)=a(x_*)+\bigo(\rho^n)$.
Therefore
\begin{align}
\left|\del_b(\det D\R^n F)_{b,p_{1,n}}\right|
&=2^nb^{2^n-1}\left|a(x_n)\del_yE_n(x_n,y_n)\right| \\
&\geq 2^n b^{2^n-1}\left||a(x_*)|-C\rho^n\right|
\end{align}
for some positive constant $C$.
Applying Proposition~\ref{prop:a-nonzero}, the result follows.
\item
Since $a$ and $E_n$ are bounded and analytic in $\Omega$ it follows from the Cauchy estimate that $\left\|\del_{x,y}\left(\del_y (aE_n)\right)\right\|=\bigo(1)$.
Therefore there exists $C>0$ such that
\begin{align}
\left\|\del_{x,y} \left(\det D\R^n F\right)_{b,p_{1,n}}\right\|
&\leq Cb^{2^n}
\end{align}

\item
By a corollary to the Mean Value Theorem, 
if $b\in [0,\bar{b})$ is not an inflection point of $f_{n,b}'(x_n(b))$ then there exist $b_0, b_1\in [0,\bar{b})$ such that
\begin{align}
\left|\del_b f_{n,b}'(x_n(b))\right|
&= \left|f_{n,b_0}'\left(x_n(b_0)\right)-f_{n,b_1}'\left(x_n(b_1)\right)\right| \cdot \left|b_0-b_1\right|
\end{align}
Since $f_{n,b}'\left(x_n(b)\right)-f_*'\left(x_*\right)=O(\rho^n)$ and $|b_0-b_1|<|\bar{b}|$, 
there exists $C_0>0$ such that the above is bounded by $C_0 |\bar{b}| \rho^n$.
Therefore, since $\left|\del_x\left(aE_n\right)\right|=\left|a'E_n+a\del_x E_n\right|=\bigo(1)$ there exists $C_1>0$ such that
\begin{align}
\left|\del_b\left(\tr D\R^nF\right)_{b,p_{1,n}}\right|
&\leq C_0\left|\bar{b}\right|\rho^n+C_12^{n}b^{2^n-1}
\end{align}
However, $\rho<3/2$ and so for $n$ sufficiently large the result follows.

\item
Since $a$ and $E_n$ are bounded and analytic in $\Omega$ it follows that $\left\|\del_{x,y}\left(\del_x (aE_n)\right)\right\|=\bigo(1)$.
We also know that $\del_{x,y}f_{n,b}=\bigo(1)$.
Therefore there exists $C>0$ such that
\begin{align}
\left\|\del_{x,y} \left(\tr D\R^n F\right)_{b,p_{1,n}}\right\|
&\leq \left\|\del_{x,y}f_{n,b}\right\|+b^{2^n}\left\|\del_{x,y}\del_x(aE_n)\right\| \\
&\leq C
\end{align}
Hence the result follows./\!/
\end{enumerate}
\begin{claim}
$\left|\del_b p_{1,n}(b)\right|<C_1b^{2^n}$
\end{claim}
\noindent
\emph{Proof of Claim:}
Differentiating the fixed-point equation 
\begin{equation}
\R^nF_{b}(p_{1,n}(b))-p_{1,n}(b)=0
\end{equation}
gives
\begin{align}
\del_b \R^nF\left(b,p_{1,n}(b)\right)+\left[\del_{x,y} \R^nF(b,p_{1,n}(b))-\id\right]\del_bp_{1,n}(b)=0
\end{align}
Since $\left|f'_*(p_{1,*})\right|\neq 1$, convergence of renormalisation implies that, 
for $n$ sufficiently large, $D\R^nF_b(p_{1,n}(b))$ has eigenvalues bounded away from $1$. 
It follows that $D\R^nF_b(p_{1,n})-\id$ is invertible.
Therefore
\begin{align}
\del_bp_{1,n}(b)
=\left(\id-\del_{x,y} \R^nF_{b,p_{1,n}(b)}\right)^{-1}\del_b \R^nF_{b,p_{1,n}(b)}
\end{align}
As $\left(\id-\del_{x,y} \R^nF\right)=\left(\id-DF_*\right)\left(1+O\left(b^{2^{n}}\right)\right)$, 
we find that $\del_bp_{1,n}(b)=\bigo\left(\del_b \R^nF_{b,p_{1,n}(b)}\right)$. 
Moreover, as $p_{1,n}(b)$ is restricted to lie on the diagonal, it follows that $\del_bp_{1,n}(b)$ is a multiple of the diagonal vector $(1,1)$./\!/

Choose a positive integer $N$ such that for all $n>N$, 
(a) $b^{2^n}>\left|\lambda_n^+\right|$ (which is possible by equation~\eqref{eq:eigenvals-asymp}) and 
(b) if $C_0$ and $C_1$ denote the constants from the previous two claims then
\begin{equation}
C_0^{-1}2^{n-1}>2b\left((3/2)^n+C_0C_1b^{2^n}\right)
\end{equation}
It now follows that for $n$ sufficiently large and $b$ sufficiently small,
\begin{align}
\left|\del_b \mathsf{D}_n\right|
&\geq 
\bigl| \,\left|\del_b\left(\det D\R^n F\right)\right|-\left\|\del_{x,y}\left(\tr D\R^n F\right)\right\|\cdot\left|\del_bp_{1,n}\right| \,\bigr| \\
&>
\left|C_0^{-1}b^{2^n-1}2^n-C_0b^{2^n}\right| \\
&> 
C_0^{-1}b^{2^n-1}2^{n-1} \\
&> 
2b^{2^n}\left((3/2)^n+C_0C_1b^{2^n}\right) \\
&> 
\left|\lambda_n^+\right| \cdot \left(\left|\del_b\left(\tr D\R^nF\right)\right|+\left\|\del_{x,y}\left(\tr D\R^nF\right)\right\|\cdot\left|\del_bp_{1,n}\right|\right) \\
&\geq 
\left|\lambda_n^+\right| \cdot \left|\del_b \mathsf{T}_n\right|
\end{align}
It follows that $\del_b\mathsf{D}_n\neq \lambda_n^+\del_b\mathsf{T}_n$ for sufficiently large $n$.
Therefore $\del_n\lambda_n^+\neq 0$ for $n$ sufficiently large, and the Proposition is shown.
%
%
\end{proof}
A simple partition of unity argument, whose proof is left to the reader, gives the following.
\begin{lem}[Interpolation Lemma]\label{lem:interpolate}
Let $r\in \{3,4,\ldots,\infty\}$.
Let $\AAA$ be a non-trivial open interval containing the point $a^*$.
Let $F\in C^1(\AAA,\mathcal I^r)$.
For each positive integer $N$ there exist 
subintervals $\AAA_1, \AAA_2\subset \AAA$ containing $a^*$,
and $F'\in C^1(\AAA_1\times \AAA_2,\mathcal I^r)$ such that for all $a_1\in \AAA_1$, $a_2\in \AAA_2$,
\begin{itemize}
\item 
$T'_{0,N-1}(a_1,a_2)=T_{0,N-1}(a_1)$; $F'(a_1,a_2)|T'_{0,N-1}(a_1,a_2)=F(a_1)|T_{0,N-1}(a_1)$
\item 
$T'_{N,*}(a_1,a_2)=T_{N,*}(a_2)$; $F'(a_1,a_2)|T'_{N,*}(a_1,a_2)=F(a_2)|T_{N,*}(a_2)$
\item 
$F'(a^*,a^*)=F(a^*)$ in $[0,1]^2$
\end{itemize}
\end{lem}
\begin{proof}[Proof of Theorem~\ref{thm:constr1}]
We will proceed by induction.
First, consider the case $d=2$.
Let $F_\mathrm{init}$ denote the one-parameter family from Proposition~\ref{prop:good-families}.
Then there exists a positive integer $N$ such that $\lambda_{m}^s$ varies regularly with  $b$ for all $m>N$.
Hence by Lemma~\ref{lem:het-tangency} there exists a parameter $b^*$ and integers $m_1$ and $n_1$ satisfying $N<m_1<n_1$ such that 
at $b=b^*$,
\begin{enumerate}
\item
$F_{\mathrm{init}}(b)$ possesses an $(m_1,n_1)$-heteroclinic tangency $q_1$,
\item
$\lambda_{m_1}^s(b)$ varies regularly with $b$.
\end{enumerate}
Choose an integer $N_1>n_1$ so that $[p_{m_1},q_1]^u(b^*)$ and $[q_1,p_{n_1}]^s(b^*)$ are disjoint from $T_{N_1,*}(b^*)$.
Then by Lemma~\ref{lem:interpolate} there exists a two-parameter family $F$ satisfying the properties (i)--(iii) and consequently property (iv).
This completes the case $d=2$.

Next, consider the case when $d\geq 3$.
Assume that there exists a $d$-parameter family $F$ satisfying the hypotheses of the theorem.
By hypothesis $F(b_1,\ldots,b_{d})|T_{N_{d-1},*}(b_1,\ldots,b_{d})$ depends only upon the parameter $b_{d}$ and in fact coincides with $F_\mathrm{init}(b_d)$.
Also by hypothesis $b(F(b_1,b_2,\ldots,b_{d}))=b_{d}$.
Once more Lemma~\ref{lem:het-tangency} implies there exists a parameter $b_{d}^*$ and integers $m_{d}$ and $n_{d}$ satisfying $N_{d-1}<m_{d}<n_{d}$ such that 
\begin{enumerate}
\item
$F(b_1,\ldots,b_{d-1},b_{d}^*)$ possesses an $(m_{d},n_{d})$-heteroclinic tangency $q_d$ for all $b_1,\ldots,b_{d-1}$
\item
$\lambda_{m_{d}}^s$ varies regularly with $b_{d}$ at $b_{d}=b_{d}^*$
\end{enumerate}
Choose an integer $N_{d}>n_{d}$ so that $[p_{m_i},q_i]^u$ and $[q_i,p_{n_i}]^s$ are disjoint from $T_{N_{d},*}$ for all $i$.
Then, as $F$ restricted to $T_{N_{d-2},*}$ is a one-parameter family, we can apply Lemma~\ref{lem:interpolate}.
This gives a $(d+1)$-parameter family which we denote by $F'(b_1,b_2,\ldots,b_d,b_{d+1})$ which satisfies 
\begin{equation}
F'(b_1,\ldots,b_{d+1})|T_{0,N_{d}-1}(b_1,\ldots,b_{d+1})
=F(b_1,\ldots,b_{d})|T_{0,N_{d}-1}(b_1,\ldots,b_{d})
\end{equation}
for all $b_1,\ldots,b_d,b_{d+1}$ on suitably restricted subintervals.
Hence properties (i)-(iii) and consequently (iv) are satisfied.
This completes the proof. 
\end{proof}

\subsection{Construction of the Second Family}\label{subsect:2ndconstr}
Next, given the resulting map, embedded in this family, with $d$ tangencies we construct a new family so that the support of each old parameter contains the support of the two new parameters.
One parameter changes the Palis invariant and the other controls a local perturbation in a neighbourhood of the corresponding point of tangency. 
Each local perturbation is chosen so that, in a neighbourhood of a point of tangency, each point on the unstable manifold moves transversely and at a controlled speed through the stable manifold as the corresponding parameter is varied.

\begin{thm}[Second Construction]\label{thm:constr2}
Let $d\geq 2$ be an integer.
Given a $d$-parameter family $F$ satisfying the hypotheses of Theorem~\ref{thm:constr1}, there exists 
an open neighbourhood $\UUU\subset \RR^{2d}$,
a family $G\in C^1(\UUU,\mathcal I^r)$ and 
a parameter $u^*\in \UUU$
such that
\begin{itemize}
\item 
$G_{u^*}=F_{b^*}$
\item 
Let $G_*=G_{u^*}$.
For $i=1,2,\ldots,d$ denote $p_{m_i}(u^*), p_{n_i}(u^*)$ and $q_i(u^*)$ by $p_{m_i}^*, p_{n_i}^*$ and $q_i^*$ respectively.
If, for $i=1,2,\ldots,d$, we define the map 
\begin{equation}
R_i=(P_{G_*,p_{m_i}^*,p_{n_i}^*},Q_{G_*,p_{m_i}^*,p_{n_i}^*,q_i^*})
\end{equation}
then $\uline{R}=(R_1,R_2,\ldots,R_d)$ is a local diffeomorphism at $u=u^*$.
\end{itemize}
\end{thm}
\begin{proof}
%
Let $F\in C^1(\BBB,\I^r)$ be as in the hypotheses of Theorem~\ref{thm:constr1}.
The following points will allow us to simplify notation.
We will construct a family $G_{u}$, where $u=(t_{1},s_{1},\ldots,t_{d},s_{d})$, so that the $i$-th pair of parameters $(t_{i},s_{i})$ correspond to a pair of local perturbations of $F$ in $T_{N_{i-1},N_{i}-1}$. 
In particular $\uline{R}$ will be a local diffeomorphism at $u=u^*$ if $R_{i}(t_{i},s_{i})$ is a local diffeomorphism at $(t_{i},s_{i})=(t_i^*,s_i^*)$ for each $i$. 
With this in mind, we drop $i$ from our notation. 
Hence we may assume we have a one-parameter family of maps, which we denote by either $F_b$ or $F(b)$ depending upon whichever is more notationally convenient, on the pair of pants $T$ so that for a fixed parameter $b^*$ the map $F_*=F(b^*)$ has a single $(m,n)$-heteroclinic tangency $q^*$ between the saddles which we denote by $p_m^*$ and $p_n^*$, for $m<n$ satisfying the properties
\begin{itemize}
\item $p_m^*, p_n^*, q^*\in T$
\item $[p_m^*,q^*]^u, [q^*,p_n^*]^s\subset T$
\end{itemize}
We also denote the preimage under $F_*$ of an object with a prime. 
For example $q'^*=F_{*}^{-1}q^*$, $p_m'^*=F_{*}^{-1}(p_m^*)$, $q''^*=F_{*}^{-1}q'^*=F_{*}^{-2}q^*$, etc..

First consider $F(b)$ at $b=b^*$.
Let $W$ be an open neighbourhood of $q(b^*)$ not intersecting any periodic orbit of $F(b)$ for any $b$. 
Let $l,\bar{l},k,\bar{k}$ be the corresponding arcs given in subsection~\ref{subsect:Palis} for the saddles $p_m^*$ and $p_n^*$ and the tangency $q^*$. 
By Remark~\ref{rmk:horiz-vert-coords} there exists a horizontal change of coordinates $\beta\colon (W,q^*)\to (\RR^2,0)$ at $q^*$.
Then we take $Q_{F_*,p_{m}^*,p_{n}^*,q^*}$ to be the function constructed in Subsection~\ref{subsect:Palis} relative to this coordinate change $\beta$.

Let $V\subset W'$ be a neighbourhood of $q'^*$. 
This also does not intersect any periodic orbit. 
Let $l',\bar{l}',k,\bar{k}'$ denote the preimages under $F_*$ of $l,\bar{l}, k$ and $\bar{k}$ respectively.
Let $\alpha\colon (V,q'^*)\to (\RR^2,0)$ be a vertical change of coordinates.
Recall that this means $\alpha(l'\cap V)\subset \{y=-|x|^a\}$ for some $a\geq 2$, and $\alpha(k'\cap V)\subset \{y=0\}$.
See figure~\ref{fig:horiz-vert} for a schematic picture.


For $j=0,\ldots,2^{m}-1$,
let $U_j\subset T$ be a neighbourhood of $F_{*}^j(p_{m}^*)$ 
which intersects $\bigcup_{j}F_{*}^j[p_{m}^*, q^*]^u$ in a single arc, 
which is disjoint from $W$, $W'$ and $\bigcup_{j}F_{*}^j[q^*,p_{n}^*]^s$, and which does not contain $q'^*=F_{*}^{-1}(q^*)$. 
By shrinking if necessary we may assume the $U_j$ have pairwise disjoint closures.
Let $U_j^0\cc U_j$ also be a neighbourhood of $F_{*}^j(p_{m})$. 
Let $U^0=\bigcup U_j^0$ and $U=\bigcup U_j$.

We now begin with our sequence of perturbations as follows. 
By similar reasoning to the Interpolation Lemma~\ref{lem:interpolate} there exists a two-parameter family $F_{b,c}$ so that $F_{b,c}|U^0=F_b$ and $F_{b,c}|\complement U=F_{c}$. Set $F^t=F_{t,b^*}$, restricting $t$ if necessary so that the orbit of $p_m$ lies in $U^0$ for all $t$.

Next, take open rectangles $S^0\cc S\subset\alpha(V)$ which contain the origin.
Take a smooth isotopy $J\colon [-s_*,s_*]\times \alpha(V)\to \alpha(V)$, where $s_*$ is sufficiently small, with support in $S$ so that 
$J_0=\id$; 
$J_s(x,y)=(x,y+s)$, for all $s$ and $(x,y)\in S^0$; 
$J_s|\del S=\id$ for all $s$; 
and vertical lines are preserved. 
Define $F^s\colon V\to F(V)$ by
\begin{equation}
F^s=F\circ\alpha^{-1}\circ J_s\circ \alpha
\end{equation}
Since the map $\alpha^{-1}\circ J_s\circ \alpha$ preserves vertical lines, is smooth and close to the identity it follows that, by restricting $s$ to a subinterval if necessary, the map $F^s$ is the restriction of a H\'enon-like map to $V$. 

We now glue together these two perturbations as follows. Let $\rho_U, \rho_V\in C^\infty([0,1]^2,\RR)$ be bump functions\footnote{Abusing terminology slightly, if neighbourhoods $W'\cc W$ are disconnected, with exactly one component of $W'$ in each component of $W$, then by the bump function for the pair $W', W$ we mean the sum of the bump functions over the connected components.} so that $\rho_U|U'=1$, $\rho_U|\complement U=0$ and $\rho_V|V'=1$, $\rho_V|\complement V=0$. 
Define
\begin{equation}
G_{t,s}=\rho_U F^t +\rho_V F^s +(1-\rho_U)(1-\rho_V)F
\end{equation}
The same argument as in the Interpolation Lemma~\ref{lem:interpolate} implies, restricting parameters if necessary, that $G_{t,s}$ is H\'enon-like for each $t$ and $s$. 
Since $G^t$ and $G^s$ are both families containing $F_{b^*}$, by an affine reparametrisation of the parameters $t$ and $s$ we can assume that $G_{0,0}=F_{b^*}$. 
It remains to show that $R\circ G_{t,s}$ is a local diffeomorphism at $(t,s)=(0,0)$. 

First, note that $P_{F_*,p_m^*,p_n^*}(G_{t,s})$ is differentiable at $(t,s)=(0,0)$.
Similarly, $Q_{F_{*},p_m^*,p_n^*,q(b^*)}(G_{t,s})$ is differentiable at $(t,s)=(0,0)$.
Hence $DR_{(t,s)}$ is well-defined at $(t,s)=(0,0)$.
Moreover, $P_{F_{*},p_m^*,p_n^*}(G_{t,s})$ is \emph{independent of the parameter $s$}.
Consequently, 
\begin{align}\label{eq:detDR}
\mathrm{det} DR_{(0,0)}
&=\del_t P_{F_*,p_m^*,p_n^*}(G_{s,t}) \del_s Q_{F_{*},p_m^*,p_n^*,q^*}(G_{s,t})|_{(t,s)=(0,0)}
\end{align}
As $F$ is a good family, so that $\lambda_m^s$ varies regularly with $t$ at $t=0$, while $\lambda_n^u$ is independent of $t$, a calculation shows that
\begin{equation}\label{eq:dPnonzero}
\del_t P_{F_*,p_m^*,p_n^*}(G_{s,t})|_{(0,0)}=\left.\frac{\del_t \lambda_m^s}{\lambda_m^s\log \lambda_n^u}\right|_{(0,0)}\neq 0
\end{equation}
It remains to show the second factor in~\eqref{eq:detDR} is non-zero.
Set $t=0$ and fix a parameter $s$. 
The following is a simple but essential observation.
\begin{claim}
Let $l_s$, $\bar{l}_s$, $k_s$, $\bar{k}_s$ denote the corresponding pieces of invariant manifold for $G_{s,0}$ and let 
$l_s', \bar{l}_s', k_s', \bar{k}_s'$ denote their preimages under $G_{s,0}$.
Then $l_s'$ is altered but $l_s$ is unchanged by varying $s$.
Similarly, $k_s'$ is unchanged but $k_s$ is altered by varying $s$.
\end{claim}
In fact the most essential piece of information is that, for all $s$,
\begin{equation}
\alpha(V\cap l_s')=\{y=0\}, \qquad \alpha(V\cap k_s')=\{y=|x|^a-s\}
\end{equation}
Therefore, let $\gamma$ denote the parametrisation of $\alpha(V\cap l_s')$ given by $\gamma(x)=(x,0)$. 
Let 
\begin{equation}
\alpha(x,y)=(\alpha_x(x),\alpha_y(x,y)), \qquad \beta=(\beta_x(x,y),\beta_y(y))
\end{equation}
Abusing notation slightly we let 
\begin{equation}
\alpha^{-1}(x,y)=(\alpha_x^{-1}(x),\alpha_y^{-1}(x,y)), \qquad \beta^{-1}(x,y)=(\beta_x^{-1}(x,y),\beta_y^{-1}(y))
\end{equation}
Consider the image of $V\cap \{y=0\}$ under the map $\beta \circ G_{0,s} \circ \alpha^{-1}$.
Then
\begin{align}
\beta \circ G_{s,0}\circ \alpha^{-1}\circ \gamma(x)
&=
\beta \circ F_{b^*} \circ \alpha^{-1} (x, s) \\
&=
(\beta_x(\phi_{b^*}(\alpha_x^{-1}(x),\alpha_y^{-1}(x,s)), \alpha_x^{-1}(x)), \beta_y (\alpha_x^{-1}(x)) )
\end{align}
\begin{figure}[htp]
\centering
\small
\psfrag{pn}{$p_n$}
\psfrag{pm}{$p_m$}
\psfrag{q}{$q$}
\psfrag{l}{$l$}
\psfrag{k}{$k$}
\psfrag{lbar}{$\bar{l}$}
\psfrag{kbar}{$\bar{k}$}
\psfrag{alpha}{$\alpha$}
\psfrag{beta}{$\beta$}
\psfrag{W}{$W$}
\psfrag{pn'}{$p_n'$}
\psfrag{pm'}{$p_m'$}
\psfrag{q'}{$q'$}
\psfrag{V}{$V$}
\psfrag{l'}{$l'$}
\psfrag{k'}{$k'$}
\psfrag{W'}{$W'$}
\psfrag{F}{$F$}
\includegraphics[width=1.0\textwidth]{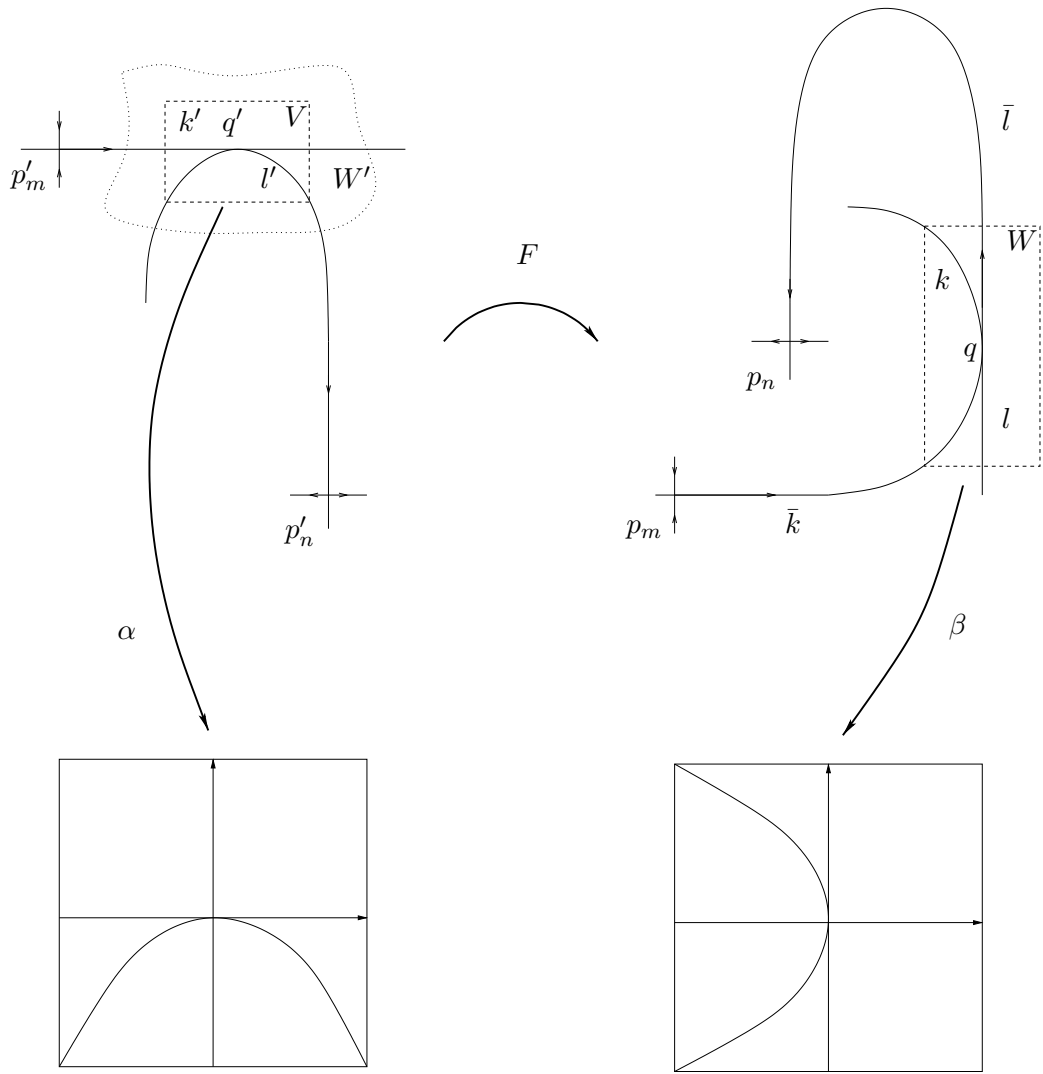}
\caption{The horizontal and vertical maps.}\label{fig:horiz-vert}
\end{figure}
Let $Y=\beta_y\circ \alpha_x^{-1}(x)$. 
Observe that this new coordinate $Y(x)$ varies regularly with $x$.  
The trace of the curve $\beta \circ G_{0,s}\circ \alpha^{-1}\circ\gamma$ coincides with that of
$\delta(Y)=(\psi_s(Y),Y)$
where
\begin{equation}
\psi_s(Y)=\beta_x(\phi_{b^*}(\beta_y^{-1}(Y),\alpha_y^{-1}(\alpha_x\beta_y^{-1}(Y),s)), \beta_y^{-1}(Y))
\end{equation}
If we let $M(Y,s)=(\beta_y^{-1}(Y),\alpha_y^{-1}(\alpha_x\beta_y^{-1}(Y),s))$ then this can then be rewritten in the form $\psi_s(Y)=\beta_x \circ F_{b^*} \circ M(Y,s)$.
Consequently 
\begin{equation}
Q_{F_*,p_m^*,p_n^*, q^*}(G_{0,s})=\psi_s(c(s))
\end{equation}
where $c(s)$ denotes the continuation of the critical point for the parameter $s=0$.
We wish to show that $\del_s (\psi_s(c(s)))\neq 0$ at $s=0$. 
Observe that by definition $\del_Y\psi |_{c(s),s}=0$.
Hence
\begin{equation}\label{eqn:d_s-phi}
\del_s(\psi_s(c(s)))|_{s=0}
=\del_s\psi|_{c(s),s}+\del_Y\psi|_{c(s),s}\del_s c|_s
=\del_s\psi|_{c(s),s}
\end{equation}
It therefore suffices to show $\del_s\psi|_{c(s),s}\neq 0$.
However, a computation shows that
\begin{align}
\del_s\psi|_{Y,s}
&=
\del_x\beta_x|_{F_{b^*}M(Y,s)}\del_y\phi_{b^*}|_{M(Y,s)}\del_y\alpha_y^{-1}|_{\alpha_x\beta_y^{-1}(Y),s}
\end{align}
Since $\alpha^{-1}$ is a diffeomorphism preserving vertical lines, $\del_y\alpha^{-1}_y\neq 0$. 
Similarly, as $\beta$ is a diffeomorphism preserving horizontal lines, $\del_x\beta_x\neq 0$. 
That $F_{b^*}$ is a diffeomorphism with Jacobian $-\del_y\phi_{b^*}\neq 0$ then implies that $\del_s\psi|_{Y,s}\neq 0$. 
By equation~\eqref{eqn:d_s-phi}, we find that $\del_s (\psi_s(c(s)))\neq 0$.
Therefore 
\begin{equation}\label{eq:dQnonzero}
\del_sQ_{F,p_m,p_n,q}(G_{s,t})|_{(0,0)}\neq 0
\end{equation}
as required. 
Equation~\eqref{eq:detDR} therefore implies, by inequalities~\eqref{eq:dPnonzero} and~\eqref{eq:dQnonzero}, that $\mathrm{det} [DR_{(0,0)}]\neq 0$ and hence $R$ is a local diffeomorphism at $(t,s)=(0,0)$. 
This completes the proof of the Proposition.
\end{proof}

\subsection{Construction of The Tangency Family}\label{subsect:tangfamily}
In this section we show, via the previous theorem, that \emph{tangency families} exist with arbitrarily many parameters.
For the definition of tangency families, see Appendix~\ref{sect:tang+full}
\newpage
\begin{cor}\label{cor:tangencyfamiliesexist}
For each integer $d\geq 1$ there exists 
\begin{itemize}
\item $\DDD\subset\RR^d$, an open neighbourhood of the origin, and $H\in C^1(\DDD,\I^r)$ which is a $d$-tangency family,
\item $\DDD'\subset\RR^d$, an open neighbourhood of the origin, and $H'\in C^1(\DDD',\H^\omega_\Omega)$ which is a $d$-tangency family. 
\end{itemize}
\end{cor}
\begin{proof}
Let $G$ denote the $2d$-parameter family of infinitely renormalisable H\'enon-like maps constructed in Theorem~\ref{thm:constr2}.
Let $\UUU'\subset \UUU$ be an open neighbourhood of the origin.
By the Weierstrass Approximation Theorem, for any $\epsilon>0$ there exists $G'\in C^\omega(\overline{\UUU'},\H^\omega_\Omega)$ such that $|G-G'|_{\overline{\UUU'}\times [0,1]^2}<\epsilon$.
As renormalisability is an open property, we can choose $\epsilon$ to be sufficiently small to ensure that $G'$ is $N_d$-times renormalisable for all parameters in $\UUU'$.

Since $\uline{R}$ is a local diffeomorphism at $u=u^*=0$ and being a local diffeomorphism is also a local property, we can assume $\epsilon$ is also small enough to ensure that $\uline{R}'$ is also a local diffeomorphism at $u=0$.
Let $\VVV$ and $\VVV'$ be open neighbourhoods of $0$ on which, respectively, $\uline{R}$ and $\uline{R}'$ are diffeomorphisms onto their images.
Let $\WWW$ and $\WWW'$ denote their respective images.

Endow $\RR^{2d}$ with the linear coordinates $P_1,Q_1,\ldots,P_d,Q_d$ and let $Q$ denote the $d$-dimensional linear subspace given by $\{Q_1=Q_2=\ldots,Q_d=0\}$.
Let
\begin{equation}
\mathcal Q=(\uline{R})^{-1}(\WWW\cap Q), \qquad
\mathcal Q'=(\uline{R}')^{-1}(\WWW\cap Q')
\end{equation}
By the Inverse Function Theorem, these are manifolds of dimension $d$ contained in $\VVV$ and $\VVV'$ respectively.
Observe that $\mathcal Q$ contains the origin.

Observe that $u^*\in\mathcal Q$ and moreover $\uline{R}(u^*)=0$.
Let  $u'^*\in\mathcal Q'$ satisfy $\uline{R}'(u'^*)=0$.
Let $\Phi\colon \mathcal U\to\RR^d$ be a chart of $\mathcal Q$ containing $u^*$ and let $\Phi'\colon \mathcal U'\to\RR^d$ be a chart of $\mathcal Q'$ containing $u'^*$. 
Assume they satisfy $\Phi(u^*)=0$ and $\Phi'(u'^*)=0$.
Let $\DDD$ and $\DDD'$ be balls contained in the respective images of these charts containing the origin.
Let $H=G\circ\Phi^{-1}|\DDD$ and $H'=G'\circ(\Phi')^{-1}|\DDD'$.
By construction, for each $i=1,2,\ldots,d$ we have $Q_i(H_{u})=0$ for any $u\in\DDD$ and $Q_i(H'_{u'})=0$ for any $u'\in\DDD'$.
Hence $H$ and $H'$ are $d$-parameter tangency families, as required.
\end{proof}
\noindent
We now prove Theorem~\ref{thm:InfPar4Henon}.
\begin{proof}[Proof of the Main Theorem]
Assume there exists a full family $F$ depending upon $d\geq 1$ parameters in either $\I^r$ for some $r$ or $\H^\omega_\Omega$.
By Corollary~\ref{cor:tangencyfamiliesexist}, for each $d\geq 1$ there exists a $d$-parameter tangency family $G$ in $\I^r$ or $\H^\omega_\Omega$ respectively. 
By Theorem~\ref{thm:tangencyimpliesnofull} the existence of a $d$-parameter tangency family contradicts the existence of a $d$-parameter full family.
Hence the Theorem follows.
\end{proof}

\begin{appendix}
\section{Tangency Families and Full Families}\label{sect:tang+full}
We show a general result for surface embeddings and diffeomorphisms which states that full families do not exist if families with persistent tangencies can be constructed.
\begin{defn}
Let $d\geq 1$ be an integer and let $\Delta\subset \RR^d$ be an arbitrary nonempty open set.
Let $M$ be an arbitrary compact manifold, possibly with boundary and
denote by $\mathcal E^r(M)$the set of orientation-preserving $C^r$-embeddings on $M$.
Let $\mathcal F\subset\mathcal E^r(M)$ be an arbitrary set.
Then $F\in C^0\left(\Delta, \mathcal E^r(M)\right)$ is a \emph{$d$-parameter full family in $\mathcal F$} if for each $f\in \mathcal F$ there exists a parameter $\uline{a}=\uline{a}(f)\in \Delta$ such that $f\sim F_{\uline{a}}$.
\end{defn}
\begin{defn}
Let $d\geq 1$ an integer.
Let $\Delta\subset \RR^d$ be an open neighbourhood which contains the origin.
Then $G\in C^1(\Delta,\mathcal F)$ is a \emph{$d$-parameter tangency family in $\mathcal F$} if
\begin{enumerate}
\item For each $i=1,\ldots,d$, $F=G_{0}$ has saddles $p_0^i$, $p_1^i$, and a heteroclinic tangency $q^i$ as in Section~\ref{subsect:Palis}, 
\item $G(\Delta)\subset \bigcap_{i=1}^d\mathcal Q_{F,p_0^1,p_1^1,q^1}$
\item $\left(P_{F,p_0^1,p_1^1}\times \ldots\times P_{F, p_0^d, p_1^d}\right)\circ G$ is a local diffeomorphism at the origin.
\end{enumerate}
\end{defn}

\begin{thm}\label{thm:tangencyimpliesnofull}
Let $d\geq 1$ be an integer and 
let $\Delta_d$ denote the unit ball in $\RR^d$.
Let $\mathcal F\subset \mathcal E^2\left([0,1]^2\right)$ be an arbitrary family of orientation-preserving diffeomorphisms.

If there exists $G_{\mathrm{tang}}\in C^1\left(\Delta_{d+1},\mathcal E^2_+\left([0,1]^2\right)\right)$, 
a $(d+1)$-parameter tangency family in $\mathcal F$ then there cannot exist $G_{\mathrm{full}}\in C^0 \left(\Delta_d,\mathcal E^2_+\left([0,1]^2\right)\right)$ which is a $d$-parameter full family in $\mathcal F$.
\end{thm}
\begin{proof}
We proceed by contradiction.
Suppose, to the contrary, there exists
\begin{equation}
G_{\mathrm{full}}\in C^0\left(\Delta_d,\mathcal E^2\left([0,1]^2\right)\right)
\end{equation}
which is a full family, for some positive integer $d$. 
Let
\begin{equation}
G_{\mathrm{tang}}\in C^1\left(\Delta_{d+1},\mathcal E^2\left([0,1]^2\right)\right)
\end{equation}
be a tangency family as defined above. 
Assume $F=G_0$ has saddles and tangencies $p_0^1,p_1^1,q^1,\ldots,p_0^{d+1},p_1^{d+1},q^{d+1}$ as above.
Denote the corresponding Palis invariant $P_{F,p_0^1,p_1^1}\times\ldots\times P_{F,p_0^{d+1},p_1^{d+1}}$ by $P$.
Without loss of generality, assume that $P\circ G_{\mathrm{tang}}$ is actually a diffeomorphism onto its image (otherwise restrict to a neighbourhood of the origin and rescale the parameter).

As $G_{\mathrm{full}}$ is full, for every $\uline{b}\in \Delta_{d+1}$ there exists $\uline{a}=\uline{a}(\uline{b})\in \Delta_d$ such that $G_{\mathrm{tang}}\left(\uline{b}\right)\sim G_{\mathrm{full}}\left(\uline{a}(\uline{b})\right)$. 
The tangency family consists of topologically inequivalent maps as they have distinct Palis invariants. 
Hence the map $\uline{a}\colon \Delta_{d+1}\to \Delta_{d}$ is injective.

Since the Palis invariant is a topological invariant,  $G_{\mathrm{tang}}\left(\uline{b}\right)\sim G_{\mathrm{full}}\left(\uline{a}(\uline{b})\right)$ implies $P\circ G_{\mathrm{tang}}\left(\uline{b}\right)=P\circ G_{\mathrm{full}}\left(\uline{a}(\uline{b})\right)$, {\it i.e.}, the following diagram commutes
\begin{equation}
\xymatrix{
\Delta_d  \ar[r]_{G_{\mathrm{full}}}           & \mathcal E^2([0,1]^2)/\sim \ar[d]^{P}  \\
\Delta_{d+1}\ar[u]_{\uline{a}}        \ar[ur]_{G_{\mathrm{tang}}} & \RR^{d+1}
}
\end{equation}
By hypothesis, $P\circ G_{\mathrm{tang}}$ is a diffeomorphism onto its image.
Hence the image $P\circ G_{\mathrm{tang}}(\Delta_{d+1})$ contains a closed $(d+1)$-dimensional ball $\Delta$.
Let $\Delta'=\left(P\circ G_{\mathrm{full}}\right)^{-1}(\Delta)$. 
Then $\Delta'$ is compact as it is closed and bounded.
Observe $\left.P\circ G_{\mathrm{full}}\right|\Delta'$ is injective as $P\circ G_{\mathrm{tang}}$ is injective and the above diagram commutes.
As $G_{\mathrm{full}}$ is continuous and $P\colon \Dom(P)\to\RR^{d+1}$ is also continuous it follows that $\left.P\circ G_{\mathrm{full}}\right|\Delta'$ is also continuous.
It then follows that $\left.P\circ G_{\mathrm{full}}\right|\Delta'$ is therefore a homeomorphism onto its image (a continuous, injective map from a compact space into a Hausdorff space is a homeomorphism onto its image).

Let $f\colon \Delta\to \Delta'$ denote the inverse of $P\circ  G_{\mathrm{full}}$. 
This is a closed map.
Moreover, as $\Delta$ is a $(d+1)$-dimensional ball, $\dim \Delta=d+1$ and since $\Delta'\subset \Delta_d$ we have $\dim \Delta'\leq d$.
Therefore~\cite[Theorem VI 7]{HandW} implies there exists a point $\uline{b}'\in \Delta'$ such that $\dim f^{-1}\left(\uline{b}'\right)\geq \dim \Delta -\dim \Delta'\geq 1$. 
Hence $f^{-1}(\uline{b}')$ must consist of more than one point, and thus $f$ cannot be a homeomorphism onto its image. 
This gives us the required contradiction.  
Therefore a full family $G_{\mathrm{full}}$ cannot exist.
\end{proof}
\end{appendix}


\begin{thebibliography}{99}

\bibitem{dCLM1} 
A.\ de Carvalho, M.\  Lyubich,  and M.\ Martens, \emph{``Renormalization in the H\'enon family I. {U}niversality but not rigidity''}, J. Stat. Phys. {\bf 121}, (2005), no. 5-6, pp. 611--669. 

\bibitem{Duf} 
E.\ Dufraine, \emph{``{\'E}quivalence topologique de champs de vecteurs en pr\'esence de connexions''}. Ph.D. Thesis, Dijon, France (2001).

\bibitem{GST}
J.\ M.\ Gambaudo, S.\ van Strien, and C.\ Tresser, \emph{``H\'enon-like maps with strange attractors: there exist $C^\infty$ Kupka-Smale diffeomorphisms on $S^2$ with neither sinks nor sources"}, Nonlinearity {\bf 2} (1989) 287--304.

\bibitem{Guckenheimer}  
J.\ Guckenheimer, in \emph{``Dynamical Systems: CIME lectures, Bressanone, Italy, 1978''}, C.I.M.E. Lectures (J. Guckenheimer, J. Moser and S. Newhouse), pp. 7--123, Progress in Mathematics {\bf 8}, (Birkhauser, New York, 2010).

\bibitem{Harrison}
J.\ Harrison, \emph{Unsmoothable Diffeomorphisms}, Ann. Math., second series, {\bf 102}, no.1, 1975, pp85--94.

\bibitem{Hartman1} 
P.\ Hartman, \emph{``Homeomorphisms of Euclidean Spaces'',}  
Bol. Soc. Mat. Mexicana (2) {\bf 5} (1960), 220--241.

\bibitem{Haz1} 
P.\ E.\ Hazard, \emph{``H\'enon-like maps with arbitrary stationary combinatorics''}. Ergod. Theory Dynam. Systems {\bf  31}, (2011), no. 5, 1391--1443.

\bibitem{HandW} 
W.\ Hurewicz and H.\ Wallman,\emph{``Dimension Theory''}, Princeton Mathematical Series {\bf 4}, (Princeton University Press, Princeton; 1948).

\bibitem{Henon} 
M.\ H\'enon, \emph{``A two-dimensional mapping with a strange attractor"}, Commun. Math. Phys. {\bf 50} (1976) 69--77.

\bibitem{LM1} 
M.\ Lyubich,  and M.\ Martens, \emph{`Renormalization in the H\'enon family II: the heteroclinic web''}. Invent. Math. {\bf 186}, (2011), no. 1, 115--189.

\bibitem{LM2} 
M.\ Lyubich,  and M.\ Martens, \emph{`Probabilistic universality in two-dimensional dynamics''}. IMS-SUNY at Stony Brook preprint 2011/02.

\bibitem{dMvS3}  
W.\ de Melo and S.\ van Strien  \emph{``Diffeomorphisms on surfaces with a finite number of moduli''},  Ergod. Th. \& Dynam. Sys., {\bf 7} (1987) 415--462.

\bibitem{dMvS} 
W.\ de Melo and S.\ van Strien, \emph{``One-dimensional Dynamics''}, Ergebnisse der Mathematik und
ihrer Grenzgebiete (3), {\bf 25}. (Springer-Verlag, Berlin-New York, 1993).

\bibitem{MilTre}  
J.\ Milnor and C.\ Tresser, \emph{"On entropy and monotonicity for real cubic maps",} (with an appendix by A. Douady and P. Sentenac) Commun. Math. Phys. {\bf 209} (2000) 123--178.

\bibitem{Newhouse} 
S. Newhouse \emph{``Non-density of Axiom A(a) on $S^2$''}, Proc. Symp. in Pure Math.{\bf 14} pp. 191--202 (1970). 

\bibitem{Palis} 
J. Palis Jr., \emph{``A differentiable invariant of topological conjugacies and moduli of stability''}, 
Ast{\'e}risque {\bf 51} (1978) 335--346.

\bibitem{RandW} 
C. Robinson and R.F. Williams, \emph{``Finite stability is not generic'',}  Dynamical systems (Proc. Sympos., Univ. Bahia, Salvador, 1971), pp. 451--462. (Academic Press, New York, 1973).

\bibitem{Sternberg1} 
S. Sternberg, \emph{``On the structure of local homeomorphisms of Euclidean n-spaces, II"}, American J. of Math., {\bf 50} (1958) 623--631.

\bibitem{Zehnder} 
E. Zehnder, \emph{``Homoclinic Points Near Elliptic Fixed Points",} Commun. Pure and Appl. Math. {\bf XXVI} (1973) 131--182. 

\end{thebibliography}
\end{document}